\documentclass[a4paper,12pt]{amsart}

\usepackage[utf8]{inputenc}
\usepackage[numbers]{natbib}
\usepackage{amsmath}
\usepackage{amssymb}
\usepackage{amsthm}
\usepackage{mathrsfs}
\usepackage{enumitem}
\usepackage{graphicx}
\usepackage{bm}
\usepackage{bbm}
\usepackage[absolute,overlay]{textpos}
\usepackage[a4paper,left=3cm,right=3cm,top=3cm,bottom=3cm]{geometry}

\newtheorem{theorem}{Theorem}[section]
\newtheorem{lemma}[theorem]{Lemma}
\newtheorem{corollary}[theorem]{Corollary}
\newtheorem{proposition}[theorem]{Proposition}

\newtheorem{assumption}{Assumption}

\theoremstyle{definition}

\theoremstyle{plain}
\newtheorem*{remark}{Remark}
\newtheorem{example}[theorem]{Example}

\newcommand{\N}{\mathbb{N}}

\newcommand{\R}{\mathbb{R}}

\newcommand{\thetanone}{\hat\theta_N}
\newcommand{\thetanonei}{\hat\theta_{N,i}}
\newcommand{\thetanoneTwenty}{\hat\theta_{20}}

\newcommand{\thetaone}{\hat\theta^\mathrm{full}_N}
\newcommand{\thetatwo}{\hat\theta^\mathrm{partial}_N}
\newcommand{\thetathree}{\hat\theta^\mathrm{linear}_N}
\newcommand{\nuone}{\hat\nu^\mathrm{full}_N}
\newcommand{\nutwo}{\hat\nu^\mathrm{partial}_N}
\newcommand{\nuthree}{\hat\nu^\mathrm{linear}_N}

\newcommand{\thetatwoone}{\hat\theta^{\mathrm{partial},1}_N}
\newcommand{\thetatwotwo}{\hat\theta^{\mathrm{partial},2}_N}

\newcommand{\doref}[1]{{\normalfont{(\ref{#1})}}}

\newcommand{\CExp}{C^\mathbb{E}}

\newcommand{\gr}{\beta}
\newcommand{\srhoexponent}{\kappa}

\newcommand{\condalpha}{\alpha > \gamma - \frac{1+\gr^{-1}}{8}}

\begin{document}

\title{Drift Estimation for Stochastic Reaction-Diffusion Systems}

\author[G. Pasemann]{Gregor Pasemann}
\author[w. Stannat]{Wilhelm Stannat}

\address{TU Berlin, Institut f\"ur Mathematik, Str. des 17. Juni 136, D-10623 Berlin, Germany}
\email{pasemann@math.tu-berlin.de}
\address{TU Berlin, Institut f\"ur Mathematik, Str. des 17. Juni 136, D-10623 Berlin, Germany}
\email{stannat@math.tu-berlin.de}


\begin{abstract}
	A parameter estimation problem for a class of semilinear stochastic evolution equations is considered. Conditions for consistency and asymptotic normality are given in terms of growth and continuity properties of the nonlinear part. Emphasis is put on the case of stochastic reaction-diffusion systems. Robustness results for statistical inference under model uncertainty are provided. 
\end{abstract}

\maketitle

We consider a semilinear stochastic partial differential equation (SPDE) 
\begin{equation}\label{eqInitial}
	\mathrm{d}X(t,x) = \theta A X(t,x)\mathrm{d}t + F(t, X(t,x))\mathrm{d}t + B\mathrm{d}W(t,x)
\end{equation}
with $X(0,x) = X_0(x)$ on a suitable domain $\mathcal{D}\subset\R^n$. Detailed conditions for the terms appearing in \doref{eqInitial} are stated in Section \ref{subsecGeneralEquation}. We write $X_t = X(t,x)$ for short. Assume that we are given complete information on the process $X$ up to a finite time $T>0$. The statistical problem we are interested in consists in estimating the unknown value $\theta>0$. \\

To this end, we adopt a maximum likelihood based approach. Denote by $X^N$ the $N$-dimensional approximation to the solution trajectory obtained by truncation in Fourier space. $X^N$ generates a probability measure on the space of continuous paths with values in $\R^N$, denoted $\mathbb{P}_\theta^N$. Of course, different values for $\theta$ lead to different measures on path space. We fix a reference parameter $\theta_0>0$ (which is arbitrary and does not necessarily coincide with the true parameter) and formally apply a version of Girsanov's theorem (as in \citep{LiptserShiryayev77}, Section 7.6.4) in order to obtain a representation for the density of $\mathbb{P}^N_\theta$ with respect to $\mathbb{P}^N_{\theta_0}$:
\begin{align*}
	\frac{\mathrm{d}\mathbb{P}^N_\theta}{\mathrm{d}\mathbb{P}^N_{\theta_0}} (X^N) = & \exp\left((\theta-\theta_0)\int_0^T\left\langle AX^N_t, (BB^T)^{-1}\mathrm{d}X^N_t\right\rangle\right. \\
		& \left. - \frac{1}{2}(\theta^2-\theta_0^2)\int_0^T|B^{-T}AX^N_t|^2\mathrm{d}t - (\theta-\theta_0)\int_0^T\langle (BB^T)^{-1}AX^N_t, F^N(t, X_t)\rangle\mathrm{d}t\right).
\end{align*}
Here, $F^N$ is the $N$-dimensional Fourier approximation of $F$. Maximizing the log-likelihood with respect to $\theta$ yields the following estimator:
\begin{equation}\label{eqFirstEstimator}
	\hat\theta_N = \frac{\int_0^T\left\langle AX^N_t, (BB^T)^{-1}\mathrm{d}X^N_t\right\rangle - \int_0^T\langle (BB^T)^{-1}AX^N_t, F^N(t, X_t)\rangle\mathrm{d}t}{\int_0^T|B^{-T}AX^N_t|^2\mathrm{d}t}.
\end{equation}
Note that the derivation of $\hat\theta_N$ 
is purely heuristic, so asymptotic properties of the estimator cannot be simply derived from the general theory of maximum likelihood estimation (as presented e.g. in \citep{IbragimovHasminskii81}). \\

The aim of this work is to extend the results from \citep{CialencoGlattHoltz11} to a class of semilinear stochastic evolution equations of the form \doref{eqInitial}. We analyze different variants of $\hat\theta_N$, which correspond to different ways of handling the nonlinear term, see Section \ref{secStatisticalInference} for details. All estimators are based on the Fourier decomposition of $X$. We present conditions concerning growth and continuity properties of the nonlinear operator $F$ which are sufficient to guarantee consistency and asymptotic normality for these estimators as the number of Fourier modes $N$ tends to infinity (see Theorem \ref{thmMain} in Section \ref{subsecMainResult}). Special emphasis is put on the important case of stochastic reaction-diffusion systems with polynomial nonlinearities. Furthermore, we study the impact of model misspecification on estimating $\theta$ in Section \ref{subsecRobustness}. More precisely: Assume that the true nonlinearity $F$ which governs the dynamics of $X$ is unknown or too complex to be handled directly. We discuss to what extent $F$ may be approximated by a simple model nonlinearity $F^\mathrm{approx}$ from the point of view of parameter estimation. Finally, we show how to adapt the argument in order to deal with a coupled system of reaction-diffusion equations, see Section \ref{secExtended}. Our motivation in this regard is to study conductance-based neuronal models, see \citep{SauerStannat16} and references therein. \\

Statistical Inference, in particular drift estimation, of stochastic ordinary differential equations (SODEs) is a well-established theory, see e.g. 
\citep{Kutoyants04, LiptserShiryayev77, LiptserShiryayev01}.
It is a well-known fact that it is in general not possible to identify the drift term of an SODE in finite time. 
The reason is that due to Girsanov's theorem the measures on path space generated by different drift terms are mutually equivalent. However, as $T\rightarrow\infty$, the true drift can be recovered asymptotically. 
The same is true for stochastic evolution equations with bounded drift on general function spaces. 
\\

Notably the situation changes for SPDEs with unbounded drift containing differential operators. In this case, it is usually possible to identify the coefficient in front of the leading term of the drift operator. This has been observed first in \citep{HubnerKhasminskiiRozovskii93} and \citep{HuebnerRozovskii95} (see also \citep{Huebner93}), and since then various publications have been devoted to studying and expanding this phenomenon, see e.g. 
\citep{Lototsky03, LototskyRozovskii99, HuebnerLototskyRozovskii97, LototskyRozovskii00}
for the case of non-diagonalizable linear equations. Notice also the recent works \citep{AltmeyerReiss19} dealing with local measurements and \citep{PospisilTribe07, BibingerTrabs17, Chong18, Chong19, CialencoHuang19, CialencoDelgadoVencesKim19} for parameter estimation under spatially and temporally discrete observations for a high-frequency regime. Surveys are presented in 
\citep{Lototsky09, Cialenco18}. 
The main focus, however, has been put on linear equations such as the stochastic heat equation, which corresponds to the case that $F$ is either zero or another linear operator. So far, only few results about parameter estimation for nonlinear SPDEs are available, most notably \citep{CialencoGlattHoltz11} (see also \citep{Cialenco18}), which considers the 2D Navier--Stokes equations and serves as a guideline for our work. \\

\section{The Model}

\subsection{General Form of the Equation}\label{subsecGeneralEquation}

Throughout this work we fix 
a final time $T>0$. Let $H$ be a Hilbert space with inner product $(\cdot,\cdot)_H$. 
Let $A$ be some negative definite self-adjoint operator on $H$ with compact resolvent and domain $D(A)\subset H$. 
We write $V=D((-A)^\frac{1}{2})$. Recall that $V\subset H\simeq H^*\subset V^*$ is a Gelfand triple, and for $h\in H$ and $v\in V$  we have ${}_{V^*}\langle h, v\rangle_V = (h, v)_H$, where ${}_{V^*}\langle\cdot,\cdot\rangle_V$ is the dual pairing between $V$ and its dual $V^*\simeq D((-A)^{-\frac{1}{2}})$.
The general model we are interested in is given by the following equation in $H$:
\begin{equation}\label{eqCompleteScrH}
	\mathrm{d}X_t = (\theta AX_t + F(t,X_t))\mathrm{d}t + B\mathrm{d}W_t,
\end{equation}
together with initial condition $X_0\in H$. Here, 
$F:[0,T]\times V\rightarrow V^*$ 
is a (possibly nonlinear) 
measurable 
operator, 
$W$ is a cylindrical Wiener process on $H$ with respect to some stochastic basis $(\Omega, \mathcal{F}, (\mathcal{F}_t)_{t\geq 0}, \mathbb{P})$,
and $B\in L_2(H)$ is of Hilbert--Schmidt type. As we need weak solutions only, 
the stochastic basis and 
the cylindrical Wiener process $W$ need not be determined in advance. 
The number $\theta>0$ is the unknown parameter to be estimated. \\

For simplicity, we restrict ourselves to the case $B = (-A)^{-\gamma}$, $\gamma > 0$. 
For later use, we introduce some notations. Let $(\Phi_k)_{k\in\N}\subset H$ be an ONB of eigenvectors of $-A$ such that the corresponding eigenvalues (taking into account multiplicity) $(\lambda_k)_{k\in\N}$ are ordered increasingly. For $N\in\N$, the projection onto the span of $\Phi_1,\dots,\Phi_N$ is called $P_N:H\rightarrow\mathrm{span}\{\Phi_1,\dots,\Phi_N\}\subset H$. The Sobolev norms on the spaces $D((-A)^\rho)\subset H$ 
will be denoted by $|x|_\rho = |(-A)^\rho x|_{H}$. The following Poincar\'e-type inequalities hold for $\rho_1 < \rho_2$:
\begin{align}
	|P_Nx|_{\rho_2} &\leq \lambda_N^{\rho_2-\rho_1}|P_Nx|_{\rho_1}, \label{eqInverseInequality}\\
	|x - P_Nx|_{\rho_1} &\leq \lambda_{N+1}^{\rho_1-\rho_2}|x-P_Nx|_{\rho_2}.\label{eqDirectInequality}
\end{align}
For our analysis, the regularity spaces
\begin{equation}
	R(\rho) := C([0,T]; D((-A)^\rho))\cap L^2([0,T]; D((-A)^{\rho+\frac{1}{2}}))
\end{equation}
will be crucial. Let $\rho\geq 0$. We say that \doref{eqCompleteScrH} has a 
{\it weak solution}\footnote{More precisely, this solution is weak in the probabilistic sense as well as in the sense of PDE theory.} in $R(\rho)$ on $[0,T]$ if there is a stochastic basis $(\Omega, \mathcal{F}, (\mathcal{F})_{t\geq 0}, \mathbb{P})$ 
together with a cylindrical Wiener process $W$ on $H$ and an $(\mathcal{F}_t)_{t\geq 0}$-adapted process $X\in R(\rho)$ such that
\begin{equation}\label{eqWeakSolution}
	X_t = X_0 + \int_0^t\left(\theta AX_s + F(s, X_s)\right)\mathrm{d}s + \int_0^tB\mathrm{d}W_s
\end{equation}
in $V^*$ a.s. for $t\in [0,T]$. We say that $X$ ``is'' a weak solution to \doref{eqCompleteScrH} if a stochastic basis and a cylindrical Wiener process can be found such that \doref{eqWeakSolution} holds. 
We need the following class of assumptions, parametrized by $\rho\geq 0$:
\begin{enumerate}
	\item[$(A_\rho)$] 
		The observed process $X$ is a weak solution to \doref{eqCompleteScrH} on $[0,T]$, unique in the sense of probability law, with $X\in R(\rho)$ a.s.
\end{enumerate}

Of course, for $(A_\rho)$ it is sufficient that \doref{eqCompleteScrH} is well-posed in the probabilistically strong sense: 
Remember that uniqueness in the sense of probability law can be inferred from pathwise uniqueness by means of the Yamada-Watanabe theorem \citep[Appendix E]{LiuRockner15}. 
We give a short and self-contained discussion on existence, uniqueness and regularity of strong solutions to \doref{eqCompleteScrH} in Appendix \ref{appWellPosed}. 

\begin{remark}
	In terms of statistical inference, it does not matter if the process we observe is a strong solution to \doref{eqCompleteScrH} in the probabilistic sense or just a weak solution. 
	The results of Theorem \ref{thmMain} below 
	depend only on the law induced by $(X_t)_{0\leq t\leq T}$ on path space (we need, of course, that this law is uniquely determined). The law of the process depends on $\theta$ but is independent of the way the weak solution is constructed. 
	We want to point out that even if the examples we are interested in are in fact constructed as strong solutions (see Theorem \ref{thmAppendixWellPosed}), this is not at all crucial from the statistical point of view. See \citep[Chapter 8]{DaPratoZabczyk14} for a discussion of weak solutions to SPDEs in the probabilistic sense.
\end{remark}

For $N\in\N$, the projected process $X^N := P_NX$ satisfies 
\begin{equation}\label{eqProcessGalerkin}
	\mathrm{d}X^N_t = (\theta AX^N_t + P_NF(t, X_t))\mathrm{d}t + P_NB\mathrm{d}W_t.
\end{equation}

Throughout this work we assume that the eigenvalues $(\lambda_k)_{k\in\N}$ of $-A$ have polynomial growth, i.e. there exist $\Lambda, \gr > 0$ such that 
\begin{equation}\label{eqLambdaAsymp}
	\lambda_k\asymp\Lambda k^\gr.
\end{equation}
In particular, $\lambda_k\asymp\lambda_{k+1}$. Here, $a_k\asymp b_k$ denotes asymptotic equivalence of two sequences of positive numbers $(a_k)_{k\in\N}$, $(b_k)_{k\in\N}$ in the sense that $\lim_{k\rightarrow\infty}\frac{a_k}{b_k}=1$. Similarly, $a_k\lesssim b_k$ means $a_k\leq C b_k$ for a constant $C>0$ independent of $k$. \\

Finally, we introduce the parameter $\rho^*$, which turns out to describe the regularity of $X$:
\begin{equation}\label{eqPreciseRegularity}
	\rho^* = \gamma - \frac{\gr^{-1}}{2}.
\end{equation}

\subsection{Statistical Inference}\label{secStatisticalInference}

We describe three estimators for $\theta$ (see \citep{CialencoGlattHoltz11}), which correspond to different levels of knowledge about the solution trajectory $(X_t)_{t\in[0,T]}$. All estimators depend on a contrast parameter $\alpha\in\R$. 
\begin{enumerate}
	\item Given continuous-time observation of the full solution $(X_t)_{t\in[0,T]}$, the heuristic derivation of the maximum likelihood estimator (cf. \citep{CialencoGlattHoltz11}) yields the following term:\footnote{Recall that $\int_0^T\langle a_t,\mathrm{d}b_t\rangle:=\int_0^Ta_t^T\mathrm{d}b_t$ for vector-valued processes $a_t$ and $b_t$.}
\begin{equation}\label{eqThetaone}
	\thetaone := - \frac{\int_0^T\langle (-A)^{1+2\alpha}X^N_t,\mathrm{d}X^N_t\rangle}{\int_0^T|(-A)^{1+\alpha}X^N_t|_H^2\mathrm{d}t} + \mathrm{bias}_N(X),
\end{equation}
where
\begin{equation}
\mathrm{bias}_N(U) := \frac{\int_0^T{}_V\langle (-A)^{1+2\alpha}X^N_t, P_NF(t, U_t)\rangle_{V^*} \mathrm{d}t}{\int_0^T|(-A)^{1+\alpha} X^N_t|_H^2\mathrm{d}t}.
\end{equation}
This estimator depends on the whole of $X$ via the bias term. Note that for $\alpha=\gamma$ this is precisely the estimator given in \doref{eqFirstEstimator}. 
\item Assume we observe just the projected solution $(X^N_t)_{t\in[0,T]}$. 
In this case, we need to replace the term $P_NF(t, X_t)$ by $P_NF(t,X^N_t)$ and consider the estimator:
\begin{equation}\label{eqThetatwo}
	\thetatwo := - \frac{\int_0^T\langle (-A)^{1+2\alpha}X^N_t,\mathrm{d}X^N_t\rangle}{\int_0^T|(-A)^{1+\alpha}X^N_t|_H^2\mathrm{d}t} + \mathrm{bias}_N(X^N).
\end{equation}
\item 
In any of the preceding observation schemes, we may leave out the nonlinear term completely:
\begin{equation}\label{eqThetathree}
	\thetathree := - \frac{\int_0^T\langle (-A)^{1+2\alpha}X^N_t,\mathrm{d}X^N_t\rangle}{\int_0^T|(-A)^{1+\alpha}X^N_t|_H^2\mathrm{d}t}.
\end{equation}
\end{enumerate}
For notational convenience, we suppress the dependence on $\alpha$ of all estimators. 

\begin{remark} \
\begin{itemize}
	\item Note that by It\^o's formula the stochastic integral in the numerator of the estimators has a robust representation:
	\begin{equation}
		\int_0^T\langle(-A)^{1+2\alpha}X^N_t,\mathrm{d}X^N_t\rangle = \frac{1}{2}\sum_{k=1}^N\lambda_k^{1+2\alpha}\left((x^k_t)^2-(x^k_0)^2 - T\lambda_k^{-2\gamma}\right),
	\end{equation}
	where $x^k:=(X,\Phi_k)_H$. Therefore, the estimators are functionals of the observed data only. 
	\item 
	Consistency of any of the three estimators as $N\rightarrow\infty$, as proven in Theorem \ref{thmMain}, implies that for $T<\infty$ the measures on $R(0)$ induced by $(X_t)_{0\leq t\leq T}$ are mutually singular for different values of $\theta$. 
	This extends the observation first made in \citep{HubnerKhasminskiiRozovskii93}. 
	\item In particular, $\theta$ can be reconstructed exactly from full spatial observation $(X_t)_{t\in[0,T]}$. This implies that $\theta$ itself is its optimal estimator in this setting. However, it is of independent interest to determine the rate and asymptotic distribution of $\thetaone$, because the analysis of the estimators $\thetatwo$ and $\thetathree$ in the case of incomplete information $(X^N_t)_{t\in[0,T]}$ is based on the results for $\thetaone$.
\end{itemize}
\end{remark}

\subsection{The Main Result}\label{subsecMainResult}

In order to state the main theorem of this paper, let us introduce some further conditions on the nonlinearity $F$, indexed by $\rho\geq 0$: 

\begin{enumerate}
	\item[$(S_\rho)$] There is $\epsilon_\rho\geq\frac{1}{2}$, 
	an integrable function $f_\rho\in L^1(0,T;\R)$ and a continuous function $g_\rho:[0,\infty)\rightarrow[0,\infty)$ such that 
		\begin{equation}\label{eqConditionSrho}
			|F(t,v)|_{\rho - \frac{1}{2}+\epsilon_\rho}^2 \leq (f_{\rho}(t) + |v|_{\rho+\frac{1}{2}}^4)g_\rho(|v|_{\rho})
		\end{equation}
		for any $t\in [0,T]$ and $v\in D((-A)^{\rho+\frac{1}{2}})$. 
\end{enumerate}
Equivalently, we may choose $g_\rho$ to be just locally bounded, because in this case there is a continuous $\tilde g_\rho:[0,\infty)\rightarrow[0,\infty)$ with $g\leq \tilde g$. We call $\epsilon_\rho$ the {\it excess regularity} of $F$.\footnote{Of course, the choice of $\epsilon_\rho$ is not unique.} A slightly different version of this condition is useful too:
\begin{enumerate}
	\item[$(S'_\rho)$] There is $\epsilon_\rho>0$, an integrable function $f_\rho\in L^1(0,T;\R)$ and a continuous function $g_\rho:[0,\infty)\rightarrow[0,\infty)$ such that 
	\begin{equation}
		|F(t, v)|_{\rho-\frac{1}{2}+\epsilon_\rho}^2 \leq (f_{\rho}(t) + |v|_{\rho+\frac{1}{2}}^2)g_\rho(|v|_\rho)
	\end{equation}
	for $t\in[0,T]$ and $v\in D((-A)^{\rho+\frac{1}{2}})$.
\end{enumerate}
Either $(S_\rho)$ or $(S'_\rho)$ is needed in order to carry out a perturbation argument with respect to the linear case. 
\begin{enumerate}
	\item[$(T_\rho)$] There is $\delta_\rho>0$ 
	and a continuous function $h_\rho:[0,\infty)^2\rightarrow[0,\infty)$ such that 
		\begin{equation}
			|F(t,u) - F(t,v)|_{\rho - \frac{1}{2}}^2 \leq h_\rho(|u|_\rho,|v|_\rho)|u-v|_{\rho+\frac{1}{2}-\delta_\rho}^{2}
		\end{equation}
		for $t\in[0,T]$ and $u,v\in D((-A)^{\rho+\frac{1}{2}})$.
\end{enumerate}
Condition $(T_\rho)$ is sufficient to formalize the intuition that $\thetatwo$ should not be worse than $\thetaone$, given that the nonlinear behavior is taken into account at least partially in the bias term. 
The next condition is required in order to ensure well-posedness of the solution to \doref{eqCompleteScrH}. In order to state the condition, we formally write $D((-A)^\infty):=\bigcap_{\rho\geq 0}D((-A)^\rho)$. 
\begin{enumerate}
	\item[$(C_\rho)$] For any $v\in D((-A)^\infty)$, the mapping $(t, u)\mapsto {}_{V^*}\langle F(t, u), v\rangle_V$ is continuous on $[0,\infty)\times D((-A)^\infty)$. Furthermore, there is a continuous function $b_\rho:[0,\infty)\rightarrow[0,\infty)$ such that
		\begin{equation}
			{}_{V^*}\langle F(t, u + v), u\rangle_{V}\leq (1 + |u|_H^2)b_\rho(|v|_{\rho+\frac{1}{2}})
		\end{equation}
		for $u\in D((-A)^\infty)$ and $v\in D((-A)^{\rho+\frac{1}{2}})$.
\end{enumerate}

Finally, we state a property, dependent on a parameter $\eta>0$, which is crucial in the examination of the estimators. However, this property results from the conditions stated above and will not be tested directly in the examples. 

\begin{enumerate}
	\item[$(R_\eta)$] It holds $X-\overline X\in R(\rho^*+\eta)$ a.s., where $\overline X_t=\int_0^tS(t-s)(-A)^{-\gamma}\mathrm{d}W_s$ is the stochastic convolution with respect to the same Wiener process that is part of the (weak) solution $X$ to \doref{eqCompleteScrH}. Here, $S$ is the strongly continuous semigroup generated by $A$ on $H$.
\end{enumerate}

We use the following two sets of conditions:

\begin{assumption}\label{asAssumption1}
	The conditions $(S_\rho)$ for $0\leq\rho<\rho^*$ and $(C_{\rho_1})$, $(T_{\rho_2})$ for some $\rho_1,\rho_2<\rho^*$ with $\delta_{\rho_2}\geq \frac{1}{2}$ are true. 
\end{assumption}

\begin{assumption}\label{asAssumption2}
	The conditions $(A_\rho)$ and $(S'_\rho)$ hold for some $\rho\in[0,\rho^*)$ such that $\rho+\epsilon_\rho > \rho^*$. 
\end{assumption}

The connection between the properties is summarized as follows:

\begin{proposition}\label{propConditionR} \
	\begin{enumerate}
		\item Under Assumption \ref{asAssumption1}, $(A_\rho)$ holds for $0\leq\rho<\rho^*$. Additionally, $(R_\eta)$ is true for every $\eta<\sup_{0\leq\rho<\rho^*}(\rho+\epsilon_\rho-\rho^*)$.
		\item Under Assumption \ref{asAssumption2}, $(A_{\rho})$ holds for every $0\leq\rho<\rho^*$, and $(R_\eta)$ is true for $\eta = \rho+\epsilon_\rho-\rho^*$.
	\end{enumerate}
\end{proposition}
The first item follows from Theorem \ref{thmAppendixWellPosed}, 
the second item is proven in Section \ref{subsecNonlinearPerturbations}. 
Recall the standing assumption $B = (-A)^{-\gamma}$ with $\gamma>0$ and that $\gr$ is given by \doref{eqLambdaAsymp}. 

\begin{theorem}\label{thmMain}
	Let either Assumption \ref{asAssumption1} or \ref{asAssumption2} be true. Let $\condalpha$.
	\begin{enumerate}
		\item The estimators 
		$\thetaone$, $\thetatwo$, $\thetathree$
		are consistent as $N\rightarrow\infty$. 
		\item $\thetaone$ is asymptotically normal. More precisely, 
		\begin{equation}\label{eqAsymptoticNormalityScheme}
			N^\frac{\gr+1}{2}(\thetaone-\theta)\rightarrow\mathcal{N}\left(0,\frac{2\theta (\gr(2\alpha-2\gamma+1)+1)^2}{T\Lambda^{2\alpha-2\gamma+1}(\gr(4\alpha-4\gamma+1)+1)}\right)
		\end{equation}
		in distribution as $N\rightarrow\infty$.\footnote{Here, $\mathcal{N}(0, V)$ denotes a normal distribution with mean zero and variance $V$.}
		\item Assume $(T_\rho)$ with parameter 
		$\delta_\rho$ for some $\rho\in[0,\rho^*)$. If $\delta_\rho > \frac{1 + \gr^{-1}}{2}$, 
		then \doref{eqAsymptoticNormalityScheme} holds with $\thetaone$ replaced by $\thetatwo$. 
		Otherwise, $N^a(\thetatwo - \theta)\xrightarrow{\mathbb{P}}0$ for each $a < \gr\delta_\rho$. 
		\item 
		For $\eta>0$ as in Proposition \ref{propConditionR}, the following is true: If $\eta > \frac{1+\gr^{-1}}{2}$, then 
		\doref{eqAsymptoticNormalityScheme} holds with $\thetaone$ replaced by either $\thetatwo$ or $\thetathree$. 
		Otherwise, $N^a(\thetatwo-\theta)\xrightarrow{\mathbb{P}} 0$ for each $a < \gr\eta$, and the same holds for $\thetathree$. 
	\end{enumerate}
\end{theorem}

\begin{remark} \
	\begin{itemize}
		\item If $X$ is a solution to the two-dimensional stochastic Navier--Stokes equations with additive noise and periodic or Dirichlet boundary conditions, we reobtain the results from \citep{CialencoGlattHoltz11}. 
		\item Note that the convergence rate and the asymptotic variance do not depend on properties of $F$. In this regard, our results are compatible with previous results on linear $F$ (see e.g. \citep{HuebnerRozovskii95, Lototsky09}) for $\alpha=\gamma$.
		\item While the conditions $(S_\rho)$, $(S'_\rho)$, $(T_\rho)$ and $(C_\rho)$ are natural conditions satisfied by a big class of examples, we do not claim that they are necessary for the conclusions of Theorem \ref{thmMain} to hold. Indeed, if $A$ and $F$ belong to a certain class of linear differential operators, \citep{HuebnerRozovskii95} and subsequent works (cf. \citep{LototskyRozovskii99, LototskyRozovskii00}) prove that an estimator of the type $\thetaone$ is consistent and asymptotically normal as $N\rightarrow\infty$ if and only if 
		\begin{equation}
			\mathrm{order}(A)\geq\frac{1}{2}(\mathrm{order}(\theta A+F)-n),
		\end{equation}
		or equivalently, $\mathrm{order}(F)\leq 2\,\mathrm{order}(A)+n$, where $n$ is the dimension of the domain. In particular, the degree of $F$ may exceed the degree of $A$. 
		\item Elementary considerations show that the asymptotic variance in \doref{eqAsymptoticNormalityScheme} is minimal for $\alpha = \gamma$, whereas the convergence rate is not affected by the choice of $\alpha$. In the ideal setting of full information that we study in this work, it is possible to reconstruct $\gamma$ and therefore also the regularity $\rho^*$ given by \doref{eqPreciseRegularity} from the observed trajectory $X^N$, e.g. via the quadratic variation of its first component at time $T$:
		\begin{equation} 
			\langle (X^N, \Phi_1)_H \rangle_T = T\lambda_1^{-2\gamma}.
		\end{equation}
		Therefore, we may set $\alpha=\gamma$ right from the beginning. If $F=0$, this corresponds to the true maximum likelihood estimator. 
		In the case of incomplete information on $\gamma$, for example time-discrete observations, which will be studied in future work, the parameter $\alpha$ can be used to ensure the divergence of the denominator of the estimators (whose expected value corresponds to the Fisher information). 
		\item Note that the asymptotic variance depends itself on the unknown parameter $\theta$. This means that in order to construct confidence intervals it is necessary to modify \doref{eqAsymptoticNormalityScheme} in a suitable way. 
		This can be done by means of a variance-stabilizing transform (see e.g. \citep[Section~3.2]{VanDerVaart98}). Alternatively, Slutsky's lemma can be used together with any of the consistent estimators for $\theta$, e.g. 
		\begin{equation}
			\frac{N^\frac{\gr+1}{2}}{\sqrt{\thetaone}}(\thetaone-\theta)\rightarrow\mathcal{N}\left(0,\frac{2 (\gr(2\alpha-2\gamma+1)+1)^2}{T\Lambda^{2\alpha-2\gamma+1}(\gr(4\alpha-4\gamma+1)+1)}\right).
		\end{equation}
		\item In general, the parameter $\delta_\rho$ from $(T_\rho)$ exceeds $\epsilon_\rho$ from $(S_\rho)$, such that a better rate for $\thetatwo$ can be guaranteed (see Section \ref{subsecReactionDiffusion}).
		\item It is possible to allow for $\omega$-dependent 
		nonlinearities $F:[0,T]\times V\times\Omega\rightarrow V^*$. In this case, it suffices to assume that $(S_\rho)$, $(S'_\rho)$, $(T_\rho)$ and $(C_\rho)$ hold almost surely in such a way that $\epsilon_\rho$ and $\delta_\rho$ are deterministic, while 
		$f_\rho$, $g_\rho$, $h_\rho$ and $b_\rho$ are allowed to depend on $\omega\in\Omega$. In particular, it is possible to extend the result to solutions of non-Markovian functional SDEs whose nonlinearity 
		depends on the whole solution trajectory $(X_t)_{t\in [0,T]}$. 
	\end{itemize}
\end{remark}

\section{Applications}\label{secApplications}

We now illustrate the general theory by means of some examples. 
We write $F(v) = F(t,v)$ whenever the nonlinearity in these examples does not depend on time explicitly. 

\subsection{The Linear Case}

For completeness, we restate the result for the purely linear case $F=0$. 
All estimators coincide, i.e. $\thetaone=\thetatwo=\thetathree$, and Theorem \ref{thmMain} reads as follows: 
\begin{corollary}
	If $\condalpha$, then
		\begin{equation}
			N^\frac{\gr+1}{2}(\thetaone-\theta)\rightarrow\mathcal{N}\left(0,\frac{2\theta (\gr(2\alpha-2\gamma+1)+1)^2}{T\Lambda^{2\alpha-2\gamma+1}(\gr(4\alpha-4\gamma+1)+1)}\right)
		\end{equation}
		in distribution as $N\rightarrow\infty$. 
\end{corollary}

\subsection{Reaction-Diffusion-Systems}\label{subsecReactionDiffusion}

In this section, we consider a bounded domain $\mathcal{D}\subset\R^n$, $n\geq 1$, with Dirichlet boundary conditions.\footnote{The argument does not depend on the boundary conditions, so Neumann- or Robin-type conditions may be used instead.} 
Set $H=L^2(\mathcal{D};\R^k)$, where $k\in\N$ is the number of coupled equations. $A$ is the Laplacian with domain $D(-A)=(H^2(\mathcal{D})\cap H^1_0(\mathcal{D}))^k$. Let $F$ be a Nemytskii-type operator on $\R^k$, i.e. $F(u)(x) = f(u(x))$ for a function $f:\R^k\rightarrow\R^k$ whose components are polynomials in $k$ variables. 
The largest degree of the component polynomials of $f$ will be denoted by $m_F$. We assume that $m_F > 1$.

\begin{example}
An important model is the SPDE
\begin{equation*}
	\mathrm{d}X_t = (\theta\Delta X_t + X_t(1-X_t)(X_t-a))\mathrm{d}t + B\mathrm{d}W_t
\end{equation*}
for $a\in(0,1)$ with nonlinearity $f(u) = u(1-u)(u-a)$. The dynamical behaviour of this equation differs significantly from a linear equation. For $a\neq\frac{1}{2}$, this equation generates travelling waves, and for $a=\frac{1}{2}$, the nonlinearity is of Allen--Cahn type, as used in phase field models. However, in terms of statistical inference on $\theta$, the nonlinear setting may be treated as a perturbation of the linear case, see Corollary \ref{corReactionDiffusionOptimal} below. 
\end{example}
\begin{proposition}\label{propReactionDiffusion} \
	\begin{enumerate}
		\item If $m_F \leq 3$ and $\rho>\frac{n}{4}-\frac{1}{2}$, then $(S_\rho)$ holds with $\epsilon_\rho=1$. 
		\item If $m_F \geq 4$ and $\rho>\frac{n}{4}-\frac{2}{m_F}$, then $(S_\rho)$ holds with $\epsilon_\rho=\frac{1}{2}+\frac{2}{m_F}$.
		\item If $m_F\leq 3$ and $n\leq 3$, then $(S_0)$ holds with $\epsilon_0=\frac{1}{2}$.
		\item If $\rho > \frac{n}{4}-\frac{1}{m_F}$, then $(S'_\rho)$ holds with $\epsilon_\rho = \frac{1}{2}+\frac{1}{m_F}$.
		\item If $\rho > \frac{n}{4}+\frac{1}{2}$, then $(T_\rho)$ holds with $\delta_\rho = 1$. 
	\end{enumerate}
\end{proposition}
\begin{proof} \
	\begin{enumerate}
		\item[(i--ii)] We have to control the term $|F(x)|_{\rho-\frac{1}{2}+\epsilon_\rho}$. Note that in order to control the norm $|\cdot|_{\rho-\frac{1}{2}+\epsilon_\rho}$, it suffices to control its one-dimensional components, so w.l.o.g. we assume $k=1$. Taking into account the triangle inequality, it suffices to control $F(x) = x^l$ for some integer $0\leq l\leq m_F$. \\
		Now $D((-A)^{\rho-\frac{1}{2}+\epsilon_\rho})\subseteq H^{2\rho-1+2\epsilon_\rho}(\mathcal{D})$ is a closed subspace, and given the choices of $\rho$ and $\epsilon_\rho$, 
		the Sobolev space $H^{2\rho-1+2\epsilon_\rho}(\mathcal{D})$ is a Banach algebra 
		\citep[p.~106]{AdamsFournier03}. Let $u\in D((-A)^{\rho+\frac{1}{2}})$. The case $l\in\{0,1\}$ is trivial. If $l\in\{2,3\}$, we have
		\begin{equation*}
			|u^l|_{\rho+\frac{1}{2}}\lesssim |u|_{\rho+\frac{1}{2}}^l.
		\end{equation*}
		This proves (i). For (ii), let $l\geq 4$. Then 
		\begin{equation}\label{eqReactionDiffusionSRho}
			|u^l|_{\rho+\frac{2}{m_F}}\lesssim |u|_{\rho+\frac{2}{m_F}}^l\lesssim |u|_{\rho+\frac{2}{l}}^l\lesssim |u|_{\rho+\frac{1}{2}}^4|u|_{\rho}^{l-4}\leq |u|_{\rho+\frac{1}{2}}^4(1+|u|_\rho)^{m_F}.
		\end{equation}
		\item[(iii)] This follows from the Sobolev embedding $H^1(\mathcal{D})\subset L^6(\mathcal{D})$ in dimension $n\leq 3$: For $l\in\{2,3\}$ and $u\in V$,
		\begin{equation*}
			|u^l|_H = |u|_{L^{2l}(\mathcal{D})}^l\lesssim |u|_{H^1(\mathcal{D})}^l\lesssim (1+|u|_V^3).
		\end{equation*}
		\item[(iv)] This is proven with a calculation similar to \doref{eqReactionDiffusionSRho}. 
		\item[(v)]
		As before, we can restrict ourselves to the case $F(x) = x^l$ with $0\leq l\leq m_F$. For $l=0$, the estimate from $(T_\rho)$ is trivial, so assume $l\geq 1$. Again using the algebra property of the Sobolev space $H^{2\rho-1}(\mathcal{D})$, 
		we have for $u,v\in D((-A)^{\rho})$:
		\begin{equation*}
			|u^l-v^l|_{\rho-\frac{1}{2}}\lesssim |u-v|_{\rho-\frac{1}{2}}\left(\sum_{i=0}^{l-1}|u|_{\rho-\frac{1}{2}}^i|v|_{\rho-\frac{1}{2}}^{l-1-i}\right)\lesssim |u-v|_{\rho-\frac{1}{2}}\left(\sum_{i=0}^{l-1}|u|_{\rho}^i|v|_{\rho}^{l-1-i}\right),
		\end{equation*}
		and the claim follows.
	\end{enumerate}
\end{proof}
\begin{remark}
	Note that the same proof allows to cover the more general case of polynomial nonlinearities whose coefficients depend on $x\in\mathcal{D}$, 
	as long as these coefficients are regular enough.
\end{remark}
Taking into account that the growth rate $\gr$ of the eigenvalues of the Laplacian is given by $\gr=\frac{2}{n}$ (see \citep{Weyl11}, or e.g. \citep[Section~13.4]{Shubin01}), we get immediately under Assumption \ref{asAssumption2}:
\begin{corollary}\label{corReactionDiffusionGeneral}
	Let $\alpha > \gamma - \frac{n+2}{16}$. If $(A_\rho)$ holds for some $\rho>\frac{n}{4}-\frac{1}{m_F}$, the estimator $\thetaone$ is asymptotically normal with rate $N^{\frac{1}{2}+\frac{1}{n}}$ and asymptotic variance $V$ given by 
	\begin{equation}
		V = \frac{2\theta(4\alpha-4\gamma+n+2)^2}{T\Lambda^{2\alpha-2\gamma+1} n(8\alpha-8\gamma+n+2)}.
	\end{equation}
	Furthermore, $\thetatwo$ and $\thetathree$ are consistent. 
\end{corollary}
\begin{remark}
	Assume that $(A_\rho)$ holds even for $\rho>\frac{n}{4}+\frac{1}{2}$. If $n\geq 2$ and $m_F\geq 3$, then the bound on the convergence rate of $\thetatwo$ due to $(T_\rho)$ is better than the bound on the convergence rate of $\thetathree$ due to $(S'_\rho)$. This corresponds to the intuition that $\thetatwo$ is "closer to the truth" than $\thetathree$.\footnote{A similar observation holds under Assumption \ref{asAssumption1} with $\epsilon_\rho$ from $(S_\rho)$.} In dimension $n=1$, $\thetatwo$ is even asymptotically normal independently of $m_F$.
\end{remark}
Loosely speaking, Corollary \ref{corReactionDiffusionGeneral} means that the estimators have good properties whenever $X$ is regular enough. Finally, we state a result (cf. \citep[Example 7.10]{DaPratoZabczyk14}) on the validity of condition $(C_\rho)$.  This allows us to make use of the better excess regularity from condition $(S_\rho)$, compared to $(S'_\rho)$, via Assumption \ref{asAssumption1}.  
\begin{proposition}
	Let $k=1$. If $m_F$ is odd and the coefficient of leading order of $f$ is negative, then $(C_\rho)$ holds for $\rho > \frac{n}{4}-\frac{1}{2}$. 
\end{proposition}
\begin{proof}
	Choose $x_0$ such that $f$ is strictly decreasing on $\R\backslash (-x_0,x_0)$, set $\mathcal{D'}:=\{|u|<x_0+|v|_\infty\}$ and $M(x):=\max(f(-x_0-x),-f(x_0+x))$. Then
	\begin{align*}
		\int_{\mathcal{D}}f(u+v)u\mathrm{d}x\leq 0 + \int_{\mathcal{D}'}f(u+v)u\mathrm{d}x\leq M(2|v|_\infty)\int_{\mathcal{D}'}|u|\mathrm{d}x\leq M(2C|v|_{\rho+\frac{1}{2}})|\mathcal{D}|(x_0+C|v|_{\rho+\frac{1}{2}}),
	\end{align*}
	where $C$ is the embedding constant of the (fractional) Sobolev space $H^{2\rho+1}(\mathcal{D})$ into $L^\infty(\mathcal{D})$ \citep[Theorem 9.8]{LionsMagenes72}.
\end{proof}
\begin{corollary}\label{corReactionDiffusionOptimal}
	Let $\gamma > \frac{n}{2}+\frac{1}{2}$, i.e. $\rho^*>\frac{n}{4}+\frac{1}{2}$. If $k=1$, $m_F$ is odd and the coefficient of leading order of $f$ is negative, then the following is true for every $\alpha > \gamma - \frac{n+2}{16}$:
	\begin{enumerate}
		\item In dimension $n=1$, 
		all three estimators are asymptotically normal whenever $m_F\leq 7$. 
		\item In dimension $n=2$, 
		$\thetaone$ is asymptotically normal and $\thetatwo$, $\thetathree$ are consistent with optimal rate whenever $m_F\leq 3$. 
	\end{enumerate}
\end{corollary}
With "consistency with optimal rate" we mean consistent with rate $N^a$ for every $a < \frac{1}{2}+\frac{1}{n}$.

\subsection{Burgers' Equation}\label{subsecBurgers}

We point out that the validity of this example has been conjectured in \citep{Cialenco18}. 
Consider the stochastic viscous Burgers equation
\begin{equation}\label{eqBurgers}
	\mathrm{d}X_t = (\nu\Delta X_t - X_t\partial_x X_t)\mathrm{d}t + B\mathrm{d}W_t
\end{equation}
on $\mathcal{D}=[0,L]$, $L>0$, with Dirichlet boundary conditions. 
Here,
\begin{equation}
	F(v) = -v\partial_x v = \partial_x\left(-\frac{1}{2}v^2\right).
\end{equation}
In this setting we have $H=L^2(\mathcal{D})$, $D(-A) = H^2(\mathcal{D})\cap H^1_0(\mathcal{D})$. \\

We follow the convention to denote the viscosity parameter by $\nu$ instead of $\theta$. Likewise, the estimators will be called $\nuone$, $\nutwo$ and $\nuthree$. 
\begin{proposition}
	The following conditions hold:
	\begin{enumerate}
		\item $(S_\rho)$ for any $\rho\geq 0$ with $\epsilon_\rho=\frac{1}{2}$. 
		\item $(T_\rho)$ for $\rho > \frac{1}{4}$ with $\delta_\rho=\frac{1}{2}$. 
		\item $(C_\rho)$ for $\rho > \frac{1}{4}$
	\end{enumerate}
\end{proposition}
\begin{proof}
	In one spatial dimension, the Sobolev space $H^s(\mathcal{D})$ is an algebra for $s>\frac{1}{2}$. So, 
	\begin{align*}
		|F(v)|_{\rho-\frac{1}{2}+\frac{1}{2}}^2
			&= \frac{1}{4}|\partial_x(v^2)|_{H^{2\rho}(\mathcal{D})}^2
			\lesssim |v^2|_{H^{2\rho+1}(\mathcal{D})}^2 
			\lesssim |v|_{H^{2\rho+1}(\mathcal{D})}^4
			= |v|_{\rho+\frac{1}{2}}^4. 
	\end{align*}
	The second property follows from the algebra property of $H^{2\rho}(\mathcal{D})$ and
	\begin{align*}
		|F(u)-F(v)|_{\rho-\frac{1}{2}}^2 \lesssim |u^2-v^2|_{\rho}^2\lesssim |u-v|_\rho^2|u+v|_\rho^2\lesssim |u-v|_\rho^2(|u|_\rho^2+|v|_\rho^2).
	\end{align*}
	Finally, for $(C_\rho)$ note that $\int_0^L(u\partial_xu)u\mathrm{d}x = \frac{1}{3}\int_0^L\partial_x(u^3)\mathrm{d}x=0$, so
	\begin{align*}
		|{}_{V^*}\langle F(u+v),u\rangle_V| &= \left|\int_0^L(u+v)\partial_x(u+v)u\mathrm{d}x\right| = \left|\int_0^L\frac{1}{2}v\partial_x(u^2)+u\partial_xvu+v\partial_xvu\mathrm{d}x\right| \\
		&\hspace{-2cm}\lesssim |\partial_xv|_\infty\int_0^Lu^2\mathrm{d}x + |\partial_xv|_\infty\int_0^Lu^2\mathrm{d}x + |v|_\infty|\partial_xv|_\infty\int_0^L|u|\mathrm{d}x \lesssim (1+|u|_H^2)(1+|v|_{\rho+\frac{1}{2}}^2),
	\end{align*}
	where we used the Sobolev embedding $H^s(\mathcal{D})\subset L^\infty(\mathcal{D})$ for $s>\frac{1}{2}$. 
\end{proof}
Similar calculations show that $(S'_\rho)$ holds for $\rho>0$ and $\epsilon_\rho=\frac{1}{4}$.
\begin{corollary}
	Assume $\gamma>\frac{1}{2}$, i.e. $\rho^*>\frac{1}{4}$. 
	Let $\alpha>\gamma-\frac{3}{16}=\rho^*+\frac{1}{16}$. 
	Then $\nuone$ is asymptotically normal with rate $N^{\frac{3}{2}}$  and asymptotic variance $V$ given by
	\begin{equation}
		V = \frac{2\theta(4\alpha-4\gamma+3)^2}{T\Lambda^{2\alpha-2\gamma+1}(8\alpha-8\gamma+3)}.
	\end{equation}
	Furthermore, $\nutwo$ and $\nuthree$ are consistent with rate $N^a$ for each $a<1$. 
\end{corollary}

\subsection{Cahn-Hillard Equation}

Let $\mathcal{D}$ be a bounded domain with smooth boundary in $\R^n$ 
for $n\leq 3$. 
Consider the equation
\begin{equation}
	\mathrm{d}X_t = (-\theta\Delta^2X_t + \Delta(X_t^3 - X_t))\mathrm{d}t + B\mathrm{d}W_t
\end{equation}
with boundary conditions $\nabla X\cdot \nu = 0$, $\nabla(\Delta X)\cdot \nu = 0$ on $\partial\mathcal{D}$, where $\nu$ is the unit normal vector pointing outwards. Set $H=L^2(\mathcal{D})$ and $V:=\{v\in H^{2}(\mathcal{D})\,|\,\nabla v\cdot \nu = 0,\,\nabla(\Delta u)\cdot \nu = 0 \textrm{ on } \partial\mathcal{D}\}$. This space is well defined, and $A:=-\Delta^2$ defines a linear operator from $V$ into $V^*$. Under the usual Riesz isomorphism $H\simeq H^*$, $V\subset H\subset V^*$ is a Gelfand triple. Finally, we set $D(A):=\{v\in V\,|\,Av\in H\}$. This equation is well-posed in the probabilistically strong sense \citep[Section 5.2]{LiuRockner15}. In particular, we have property $(A_0)$. Note that $A$ is a differential operator of order four, which means that $D((-A)^\rho)\subset H^{4\rho}(\mathcal{D})$. This differs from the situation in the previous examples.

\begin{proposition} \
	\begin{enumerate}
		\item If $n\in\{1,2\}$, $(S'_\rho)$ is true for $\rho\geq 0$ with $\epsilon_\rho=\frac{1}{3}$.
		\item If $n=3$, then $(S'_0)$ is true with $\epsilon_0=\frac{1}{4}$, and $(S'_\rho)$ is true for $\rho>\frac{1}{24}$ with $\epsilon_\rho=\frac{1}{3}$
		\item $(T_\rho)$ is true for $\rho\geq \frac{n}{8}$ with $\delta_\rho=\frac{1}{2}$
	\end{enumerate}
\end{proposition}
\begin{proof}
	We just prove the first statement from (ii). The remaining calculations are similar to those in Proposition \ref{propReactionDiffusion} and are omitted here. Let $u\in V$ and $v\in D((-A)^\frac{1}{4})\subset H^1(\mathcal{D})$. By integration by parts, 
	\begin{align*}
		|\langle\Delta(u^3-u),v\rangle|\lesssim|\nabla u|_H|\nabla v|_H|3u^2-1|_\infty\lesssim|u|_{H^1(\mathcal{D})}|v|_{H^1(\mathcal{D})}(1+|u|_\infty^2).
	\end{align*}
	We use $|u|_\infty\lesssim |u|_{H^2(\mathcal{D})}^\frac{3}{4}|u|_{L^2(\mathcal{D})}^\frac{1}{4}$ \citep[Theorem 5.8]{AdamsFournier03} and $|u|_{H^1(\mathcal{D})}\lesssim |u|_{H^2(\mathcal{D})}^\frac{1}{2}|u|_{L^2(\mathcal{D})}^\frac{1}{2}$ in order to obtain 
	\begin{align*}
		|\Delta(u^3-u)|_{\rho-\frac{1}{4}}\lesssim (1 + |u|_{\rho+\frac{1}{2}}^2)(1+|u|_{\rho})
	\end{align*}
	for $\rho=0$. 
\end{proof}
Note that the eigenvalues $\lambda_k$ grow like $k^\frac{4}{n}$ \citep[Section 13.4]{Shubin01}. 

\begin{corollary}
	Choose $\alpha > \gamma - \frac{n+4}{32}$. 
	Then $\thetaone$ is asymptotically normal with rate $N^{\frac{1}{2}+\frac{2}{n}}$ and asymptotic variance $V$ given by
	\begin{equation}
		V = \frac{2\theta(8\alpha-8\gamma+n+4)^2}{T\Lambda^{2\alpha-2\gamma+1}n(16\alpha-16\gamma+n+4)}.
	\end{equation}
	If $n\in\{1,2\}$, let $\gamma$ (thus $\rho^*$) be arbitrary, if $n=3$, let $\gamma>\frac{10}{24}$, i.e. $\rho^*>\frac{1}{24}$. Then $\thetatwo$, $\thetathree$ are consistent with rate $N^a$ for $a < \frac{4}{3n}$. If $\rho^*>\frac{n}{8}$, then we can choose even $a<\frac{2}{n}$ for $\thetatwo$. 
\end{corollary}

\subsection{Robustness under Model Uncertainty}\label{subsecRobustness}

In the preceding examples we assumed that the dynamical law of the process we are interested in is perfectly known. However, it may be reasonable to consider the case when this is not true. We may formalize such a partially unknown model as
\begin{equation}\label{eqUncertainEquation}
	\mathrm{d}X_t = (\theta\Delta X_t + F(t,X_t) + G(t,X_t))\mathrm{d}t + B\mathrm{d}W_t,
\end{equation}
where $G:[0,T]\times V\rightarrow V^*$ is an unknown 
perturbation. We assume that the model is well-posed (i.e. $(A_\rho)$ holds for $0\leq\rho<\rho^*$) and that $F$ satisfies $(S'_\rho)$ with $\rho+\epsilon_\rho>\rho^*$. Let $\thetaone$, $\thetatwo$ and $\thetathree$ be given by the same terms as before, i.e. $\thetaone$ and $\thetatwo$ include knowledge on $F$ but not on $G$. 
\begin{proposition}
	If $G$ satisfies $(S'_\rho)$ with $\rho+\epsilon_\rho>\rho^*$, then $\thetaone$, $\thetatwo$ and $\thetathree$ are consistent. 
\end{proposition}
This follows directly from the discussion in 
Section \ref{secProof}, 
taking into account the decomposition
\begin{equation}
	\thetaone - \theta = -\frac{\int_0^T\langle(-A)^{1+2\alpha}X^N_t,P_NB\mathrm{d}W_t\rangle}{\int_0^T|(-A)^{1+\alpha} X^N_t|_H^2\mathrm{d}t} - \mathrm{bias}^G_N(X)
\end{equation}
with
\begin{equation}
	\mathrm{bias}^G_N(X) = \frac{\int_0^T{}_V\langle (-A)^{1+2\alpha}X^N_t, P_NG(t,X_t)\rangle_{V^*} \mathrm{d}t}{\int_0^T|(-A)^{1+\alpha} X^N_t|_H^2\mathrm{d}t}
\end{equation}
and similar decompositions for $\thetatwo$ and $\thetathree$. \\

It is easy to verify that if $(S'_\rho)$ holds for $F$ and $G$ separately with excess regularity $\epsilon^F_\rho$ resp. $\epsilon^G_\rho$, then a version of $(S'_\rho)$ holds for $F+G$ as well, with excess regularity $\min(\epsilon^F_\rho,\epsilon^G_\rho)$. However, in general the excess regularity $\epsilon^{F+G}_\rho$ of $F+G$ can be chosen larger due to cancellation effects of $F$ and $G$. 
\begin{corollary}\label{corRobust} \
	\begin{enumerate}
		\item If $\rho+\epsilon^G_\rho-\rho^* > \frac{1+\gr^{-1}}{2}$, then $\thetaone$ is asymptotically normal with rate $N^\frac{\gr+1}{2}$.
		\item If $\rho+\epsilon^G_\rho-\rho^* > \frac{1+\gr^{-1}}{2}$ and $F$ satisfies $(T_\rho)$ with $\delta_\rho > \frac{1+\gr^{-1}}{2}$, then $\thetatwo$ is asymptotically normal with rate $N^\frac{\gr+1}{2}$.
		\item If $\rho+\epsilon^{F+G}_\rho-\rho^*>\frac{1+\gr^{-1}}{2}$, then asymptotic normality with rate $N^\frac{\gr+1}{2}$ carries over to all estimators. 
	\end{enumerate}
\end{corollary}
Said another way, the excess regularity of $G$ determines essentially to what extent the results from Theorem \ref{thmMain} remain valid. A large value for $\epsilon^G_\rho$ corresponds to a small perturbation. 

\begin{remark}\
	\begin{itemize}
		\item In applications it is common to approximate a complicated nonlinear system by its linearization. From this point of view, the case that $F$ itself is linear 
		in \doref{eqUncertainEquation} becomes relevant. Of course, it is desirable to maintain the statistical properties of the linear model under a broad class of nonlinear perturbations. 
		\item It is possible to interpret the nonlinear perturbation as follows: Assume there is a true nonlinearity $F^\mathrm{true}$ describing the model precisely. Assume further that we either do not know the form of $F^\mathrm{true}$ or we do not want to handle it directly due to its complexity. Instead, we approximate $F^\mathrm{true}$ by some nonlinearity $F = F^\mathrm{approx}$ which we can control. If our approximation is good (in the sense that $(S'_\rho)$ holds for $G=F^\mathrm{true}-F^\mathrm{approx}$ with suitable excess regularity), then the quality of the estimators which are merely based on the approximating model can be guaranteed, i.e. they are consistent or even asymptotically normal. 
		The approximating quality of $F^\mathrm{approx}$ is measured by the excess regularity of $G$. 
		\item As $G$ is unknown, no knowledge of $G$ can be incorporated into the estimators, and 
		condition $(T_\rho)$ need not be required to hold for $G$. 
		\item The previous examples show that $(S'_\rho)$ is fulfilled for a broad class of nonlinearities $G$ (assuming that $\rho$ is sufficiently large if necessary). 
	\end{itemize}
\end{remark}

\section{Numerical Simulation}

We simulate the Allen--Cahn equation
\begin{equation*}
	\mathrm{d}X_t = (\theta\Delta X_t + X_t - X_t^3)\mathrm{d}t + (-\Delta)^{-\gamma}\mathrm{d}W_t
\end{equation*}
on $[0,1]$ with Dirichlet boundary conditions and initial condition $X_0(x) = \sin(\pi x)$. We discretize the equation in Fourier space and simulate $N_0 = 100$ modes with a linear-implicit Euler scheme with temporal stepsize $h_\mathrm{temp} = 2.5\times 10^{-5}$ up to time $T=1$. The spatial grid is uniform with mesh $h_\mathrm{space} = 5\times 10^{-4}$. The true parameter is $\theta=0.02$. We have run $M=1000$ Monte-Carlo simulations for each of the choices $\gamma=0.4$ and $\gamma=0.8$. In any case, we have set $\alpha=\gamma$. Remember that in this setting all estimators are asymptotically normal.\\

\begin{figure}[h]
	\includegraphics[width=\textwidth]{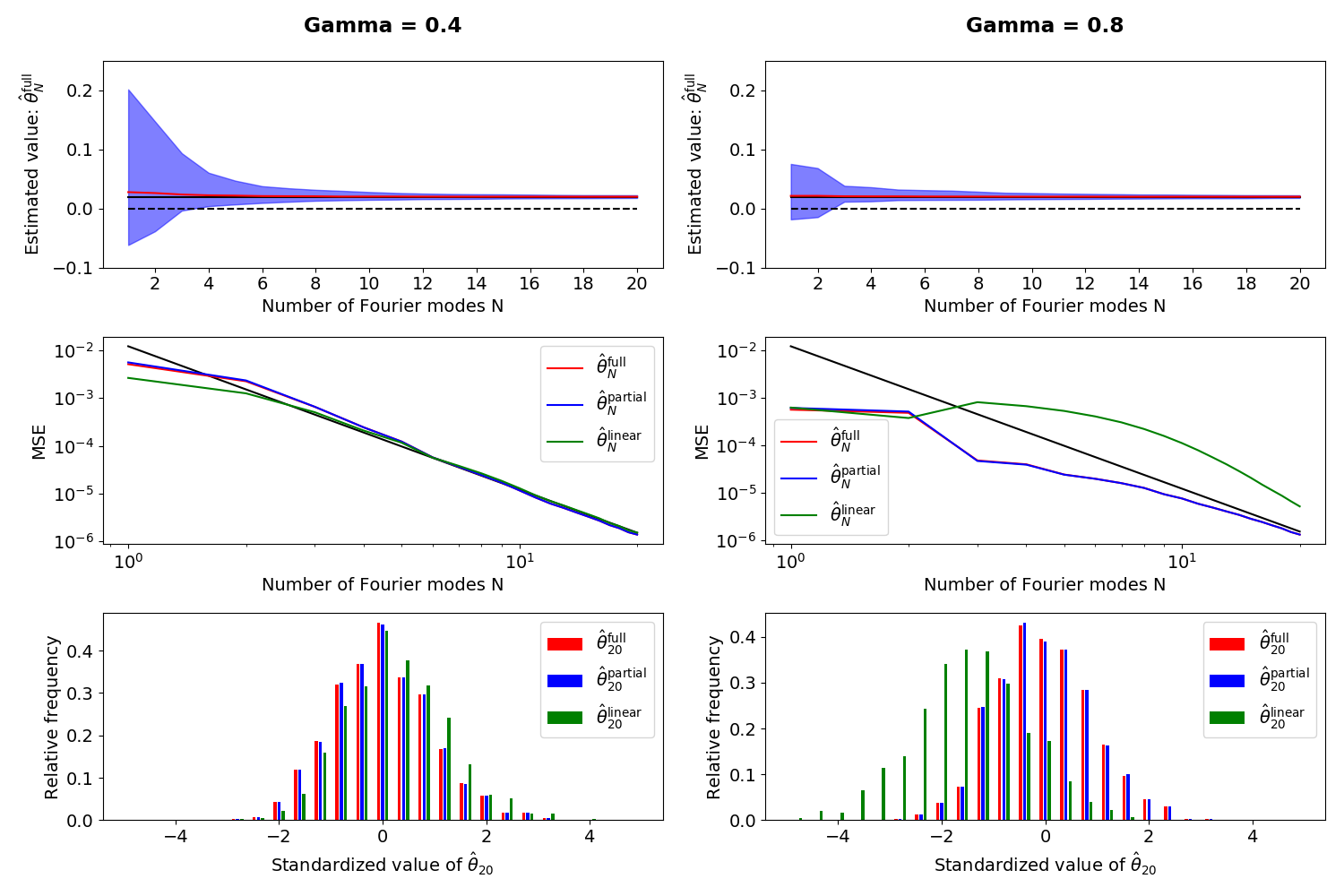}
	\caption{The left column corresponds to the case $\gamma=0.4$, the right column to the case $\gamma=0.8$. {\it First row:} The median (red) and the $2.5$-percentile as well as the $97.5$-percentile (boundaries of the blue region) of $M=1000$ Monte-Carlo simulations of $\thetaone$ are plotted. The solid black line represents the true parameter $\theta=0.02$, the dashed line is plotted at zero. {\it Second row:} The mean squared error (MSE), given by the term $\frac{1}{M}\sum_{i=1}^M(\thetanonei-\theta)^2$, is plotted against the squared rate function $N\mapsto(\sqrt{V}N^{-\frac{3}{2}})^2$, where $V$ is the asymptotic variance from Theorem \ref{thmMain} and $\thetanone$ any of the three estimators. {\it Third row:} Histogram of the standardized values of $\thetanoneTwenty$, i.e. the values of $\frac{N^\frac{3}{2}}{\sqrt{V}}(\thetanone-\theta)$ for $N=20$. Each bin has a width of $0.4$. Outliers outside the range $[-5,5]$ are put into the leftmost (or rightmost, resp.) bin.}\label{figPlot}
\end{figure}

Figure \ref{figPlot} illustrates consistency, the convergence rate and the asymptotic distribution from Theorem \ref{thmMain}. 
As expected, the values of $\thetaone$ and $\thetatwo$ are closer to each other than to $\thetathree$. Note that the quality of $\thetathree$ in this simulation depends on the level of noise given by $\gamma$, with decreasing accuracy under smooth noise. Our interpretation is that the nonlinearity becomes more highlighted if the noise is less rough. \\

We mention that for simulations with even larger values of $\gamma$ (take $\gamma=1.3$), the values of $\thetathree$ are mostly negative and therefore not related to the true parameter, while $\thetaone$ and $\thetatwo$ stay consistent. Of course, this effect may be influenced by the number of Fourier modes $N_0$ used for the simulation.

\section{Proof of Theorem \ref{thmMain}}
\label{secProof}

We follow closely the arguments which have been given in \citep{CialencoGlattHoltz11} for the special case of the Navier--Stokes equations in two dimensions. 
Using a slightly different version of the central limit theorem (CLT) for local martingales, we obtain a direct proof of the asymptotic normality for $\thetaone$.

\subsection{Properties of the Linear Process}
\label{subsecTheLinearCase}

First, we recall briefly some results for 
the case $F=0$. 
Consider the linear equation 
\begin{equation}\label{eqLinear}
	\mathrm{d}\overline X_t = \theta A\overline X_t\mathrm{d}t + B\mathrm{d}W_t,\quad \overline X_0 = 0,
\end{equation}
where $B = (-A)^{-\gamma}$. 
We define $\overline x^k:=(\overline X,\Phi_k)_H$. Then the $\overline x^k$ are independent one-dimensional Ornstein--Uhlenbeck processes 
\begin{equation}
	\mathrm{d}\overline x^k_t = -\theta\lambda_k\overline x^k_t\mathrm{d}t + \lambda_k^{-\gamma}\mathrm{d}W^k_t,\quad \overline x^k_0 = 0,
\end{equation}
where $(W^k)_{k\in\N}$ are independent one-dimensional Brownian motions, and the solutions have the explicit representation
\begin{equation}
	\overline x^k_t = \lambda_k^{-\gamma}\int_0^te^{-\theta\lambda_k(t-s)}\mathrm{d}W^k_s.
\end{equation}
\begin{lemma}[cf. \citep{CialencoGlattHoltz11, Lototsky09}]\label{lemLinearModeAsymp}
	It holds that
	\begin{enumerate}
		\item $\mathbb{E}\int_0^T(\overline x^k_t)^2\mathrm{d}t\asymp \frac{T\Lambda^{-(2\gamma+1)}}{2\theta}k^{-\gr(2\gamma+1)}$ and
		\item $\mathrm{Var}\int_0^T(\overline x^k_t)^2\mathrm{d}t\asymp\frac{T\Lambda^{-(4\gamma+3)}}{2\theta^3}k^{-\gr(4\gamma+3)}$
	\end{enumerate}
	as $k\rightarrow\infty$. 
\end{lemma}
\begin{proof}[Sketch of proof]
	Use that $\overline x^k_s$ and $\overline x^k_t$, $s\leq t$, are jointly Gaussian with mean zero and
	\begin{equation*}
		\mathbb{E}[\overline x^k_s\overline x^k_t] = \frac{\lambda_k^{-(2\gamma+1)}}{2\theta}(e^{-\theta\lambda_k(t-s)} - e^{-\theta\lambda_k(t+s)}).
	\end{equation*}
	Now $(i)$ follows with the help of $\mathbb{E}\int_0^T(\overline x^k_t)^2\mathrm{d}t = \int_0^T\mathbb{E}(\overline x^k_t)^2\mathrm{d}t$. 
	For $(ii)$, use 
	\begin{equation*}
		\mathbb{E}[(\overline x^k_s)^2(\overline x^k_t)^2] = \mathrm{Var}[\overline x^k_s]\mathrm{Var}[\overline x^k_t] + 2\mathrm{Cov}(\overline x^k_s, \overline x^k_t)^2
	\end{equation*}
	and $\mathbb{E}\left(\int_0^T(\overline x^k_t)^2\mathrm{d}t\right)^2 = 2\int_0^T\int_0^t\mathbb{E}[(\overline x^k_s)^2(\overline x^k_t)^2]\mathrm{d}s\mathrm{d}t$.
\end{proof}
We write $\overline X^N := P_N\overline X$. By multiplying the asymptotic representations from Lemma \ref{lemLinearModeAsymp} with $\Lambda^{2\alpha+2}k^{\gr(2\alpha+2)}$ and $\Lambda^{4\alpha+4}k^{\gr(4\alpha+4)}$, respectively, and summing up to index $N$, we obtain the following cumulative version if $\alpha > \gamma-\frac{1+\gr^{-1}}{2}$:
\begin{equation}\label{eqMeanCumulativeAsymptotics}
	\mathbb{E}\int_0^T|(-A)^{1+\alpha}\overline X^N_t|_H^2\mathrm{d}t\asymp \CExp_\alpha N^{\gr(2\alpha-2\gamma+1) + 1},
\end{equation}
and if $\alpha > \gamma-\frac{1+\gr^{-1}}{4}$:
\begin{equation}
	\mathrm{Var}\int_0^T|(-A)^{1+\alpha}\overline X^N_t|_H^2\mathrm{d}t\asymp C^\mathrm{Var}_\alpha N^{\gr(4\alpha-4\gamma+1)+1},
\end{equation}
where
\begin{equation*}
	\begin{matrix}
	\CExp_\alpha &=& \CExp(\theta, T,\Lambda, \gr, \gamma, \alpha) &=& \frac{T\Lambda^{2\alpha-2\gamma+1}}{2\theta (\gr(2\alpha-2\gamma+1)+1)}, \\
	C^\mathrm{Var}_\alpha &=& C^\mathrm{Var}(\theta, T,\Lambda, \gr, \gamma, \alpha) &=& \frac{T\Lambda^{4\alpha-4\gamma+1}}{2\theta^3(\gr(4\alpha-4\gamma+1)+1)}.
	\end{matrix}
\end{equation*}
\begin{lemma}\label{lemFisherAsymptoticsLinear}
	Let $\alpha > \gamma-\frac{1+\gr^{-1}}{4}$. 
	Then
	\begin{equation}
		\frac{\int_0^T|(-A)^{1+\alpha}\overline X^N_t|_H^2\mathrm{d}t}{\mathbb{E}\int_0^T|(-A)^{1+\alpha}\overline X^N_t|_H^2\mathrm{d}t}\rightarrow 1
	\end{equation}
	as $N\rightarrow\infty$ in probability.\footnote{Using the strong law of large numbers \citep{Shiryaev96}, one can easily show that the convergence holds even almost surely, see \citep{CialencoGlattHoltz11}.}
\end{lemma}
\begin{proof}
Taking into account asymptotic equivalence, we obtain
	\begin{align*}
		\mathbb{P}\left(\left|\frac{\int_0^T|(-A)^{1+\alpha}\overline X^N_t|_H^2\mathrm{d}t}{\mathbb{E}\int_0^T|(-A)^{1+\alpha}\overline X^N_t|_H^2\mathrm{d}t}-1\right| > \epsilon\right)
			&\leq \frac{\mathrm{Var}\int_0^T|(-A)^{1+\alpha}\overline X^N_t|_H^2\mathrm{d}t}{\epsilon^2\left(\mathbb{E}\int_0^T|(-A)^{1+\alpha}\overline X^N_t|_H^2\mathrm{d}t\right)^2} \\
			&\lesssim \frac{N^{\gr(4\alpha-4\gamma+1)+1}}{N^{\gr(4\alpha-4\gamma+2)+2}} = N^{-(\gr+1)},
	\end{align*}
	which goes to zero as $N\rightarrow\infty$.
\end{proof}

We close this section by giving the precise regularity for the linear process $\overline X$.

\begin{proposition}\label{propLinearRegularity}
	The unique solution $\overline X$ to \doref{eqLinear} satisfies $\overline X\in C([0,T];D((-A)^{\rho+\frac{1}{2}}))$ a.s. for $\rho < \rho^*$, where 
	\begin{equation}
		\rho^* = \gamma - \frac{\gr^{-1}}{2}.
	\end{equation}
	In particular, $\overline X\in R(\rho)$ for these $\rho$. Conversely, $\overline X\notin R(\rho^*)$ a.s.
\end{proposition}
\begin{proof}
	Let $\rho < \rho^*$. It suffices to prove that $\overline Y := (-A)^{\rho+\frac{1}{2}}\overline X \in C([0,T]; H)$. Given that $\overline Y$ satisfies
	\begin{equation}\label{eqYRescaled}
		\mathrm{d}\overline Y_t = \theta A \overline Y_t\mathrm{d}t + (-A)^{\rho+\frac{1}{2}-\gamma}\mathrm{d}W_t, \quad \overline Y_0 = 0,
	\end{equation}
	this follows from the factorization formula \citep[Section 5.3]{DaPratoZabczyk14} once we know that 
	\begin{equation}\label{eqHilbertSchmidtContinuous}
		\int_0^1t^{-2\delta}|S(t)(-A)^{\rho+\frac{1}{2}-\gamma}|_{HS}^2\mathrm{d}t<\infty
	\end{equation}
	for some $\delta\in(0,\frac{1}{2})$, where $S$ is the strongly continuous semigroup generated by $A$ and $|\cdot|_{HS}$ denotes the Hilbert-Schmidt norm.\footnote{More precisely: \doref{eqHilbertSchmidtContinuous} with $\delta=0$ yields the existence of a unique solution in $H$ to \doref{eqYRescaled}, and if $\delta$ can be chosen in $(0,\frac{1}{2})$, this solution is continuous in time.} Indeed:
	\begin{align*}
		\int_0^1t^{-2\delta}|S(t)(-A)^{\rho+\frac{1}{2}-\gamma}|_{HS}^2\mathrm{d}t &= \sum_{k=1}^\infty \int_0^1t^{-2\delta}e^{-2\theta\lambda_kt}\lambda_k^{2\rho+1-2\gamma}\mathrm{d}t \\
		&\leq \sum_{k=1}^\infty \lambda_k^{2\rho+1-2\gamma}\int_0^\infty \left(\frac{s}{2\theta\lambda_k}\right)^{-2\delta}e^{-s}\frac{1}{2\theta\lambda_k}\mathrm{d}s \\
		&= \sum_{k=1}^\infty\lambda_k^{2\rho-2\gamma+2\delta}(2\theta)^{2\delta-1}\Gamma(1-2\delta) \lesssim \sum_{k=1}^\infty k^{\gr(2\rho-2\gamma+2\delta)}.
	\end{align*}
	Here, $\Gamma$ is the Gamma function. The last sum is finite if $\gr(2\rho-2\gamma+2\delta) < -1$, i.e. $\rho < \rho^*-\delta$. Now $\delta\in (0,\frac{1}{2})$ can be chosen sufficiently small. 
	Conversely, the discussion leading to \doref{eqMeanCumulativeAsymptotics} shows that
	\begin{equation}
		\mathbb{E}\int_0^T|(-A)^{\rho^*+\frac{1}{2}}\overline X^N_t|_{H}^2\mathrm{d}t\rightarrow\infty
	\end{equation}
	as $N\rightarrow\infty$, and by Lemma \ref{lemFisherAsymptoticsLinear}, we have even a.s. pathwise divergence (take an a.s. converging subsequence in the statement of the Lemma). Hence $\overline X\notin L^2([0,T];D((-A)^{\rho^*+\frac{1}{2}}))$ almost surely, and the claim follows.
\end{proof}

\subsection{Asymptotic Behaviour in the Semilinear Case}
\label{subsecNonlinearPerturbations}

\subsubsection{Proof of Proposition \ref{propConditionR} (ii)}

Assuming that $(A_\rho)$ and $(S'_\rho)$ hold for some $\rho\in[0,\rho^*)$, we define $\widetilde X := X - \overline X$ and $\widetilde X^N := P_N\widetilde X$, where as before $\overline X$ is the solution to \doref{eqLinear}. These processes are well-defined and satisfy
\begin{equation}\label{eqNonlinearGalerkinProjection}
	\widetilde X^N_t = X^N_0 + \int_0^t(\theta A\widetilde X^N_s + F(s, \overline X_s + \widetilde X_s))\mathrm{d}s.
\end{equation}

Calculations similar to those in Lemma \ref{lemBoundsGalerkin} show\footnote{Note, however, that the Galerkin approximants to $\widetilde X$ we use in this section are not identical to the approximants from Lemma \ref{lemBoundsGalerkin}, which satisfy \doref{eqNonlinearGalerkin} rather than \doref{eqNonlinearGalerkinProjection}}
\begin{equation}
	\sup_{0\leq t\leq T}|\widetilde X^N_t|_{\rho+\epsilon_\rho}^2 + \theta\int_0^T|\widetilde X^N_t|_{\rho+\frac{1}{2}+\epsilon_\rho}^2\mathrm{d}t \leq |X^N_0|_{\rho+\epsilon_\rho}^2 + C_\theta\int_0^T|P_NF(t, \overline X_t + \widetilde X_t)|_{\rho-\frac{1}{2}+\epsilon_\rho}^2\mathrm{d}t.
\end{equation}
The nonlinear term is estimated as follows:
\begin{align*}
	\int_0^T|P_NF(t, \overline X_t + \widetilde X_t)|_{\rho-\frac{1}{2}+\epsilon_\rho}^2\mathrm{d}t & \leq \int_0^T(f_\rho(T)+|\overline X_t+\widetilde X_t|_{\rho+\frac{1}{2}}^2)g_\rho(|\overline X_t+\widetilde X_t|_\rho)\mathrm{d}t \\
	& \leq \sup_{0\leq t\leq T}g_\rho(|\overline X_t+\widetilde X_t|_\rho)\int_0^T\left(f_\rho(t)+|\overline X_t+\widetilde X_t|_{\rho+\frac{1}{2}}^2\right)\mathrm{d}t < \infty,
\end{align*}
so $(\widetilde X^N)_{N\in\N}$ is bounded in $R(\rho+\epsilon_\rho)$, thus $\widetilde X\in R(\rho+\epsilon_\rho)$. We have proven:
\begin{lemma}
	If $(A_\rho)$ and $(S'_\rho)$ hold for some $\rho\in[0,\rho^*)$, then $\widetilde X\in R(\rho+\epsilon_\rho)$ a.s.
\end{lemma}
We finish the proof of Proposition \ref{propConditionR} (ii) with the following Lemma:
\begin{lemma}\label{lemRegularityBoundFullProcess}
If $(A_\rho)$ and $(S'_\rho)$ hold for some $\rho\in[0,\rho^*)$ with $\rho+\epsilon_\rho > \rho^*$, 
then almost surely $X\in R(\rho^*-\delta)$ for $\delta>0$ and $X\notin R(\rho^*)$. 
\end{lemma}
\begin{proof}
	This follows from $\overline X\in R(\rho^*-\delta)$ for $\delta > 0$, $\overline X\notin R(\rho^*)$ and $\widetilde X\in R(\rho^*)$ almost surely. 
\end{proof}

\subsubsection{An Asymptotic Growth Property}

\begin{proposition}\label{propFisherAsymptoticsNonlinear}
	Assume that $(R_\eta)$ holds for some $\eta>0$. 
	Let $\alpha > \gamma - \frac{1+\gr^{-1}}{4}$. Then
	\begin{equation}
		\frac{\int_0^T|(-A)^{1+\alpha}X^N_t|_H^2\mathrm{d}t}{\mathbb{E}\int_0^T|(-A)^{1+\alpha}\overline X^N_t|_H^2\mathrm{d}t}\rightarrow 1
	\end{equation}
	as $N\rightarrow\infty$ in probability.\footnote{As in Lemma \ref{lemFisherAsymptoticsLinear}, almost sure convergence holds in fact.}
\end{proposition}
\begin{proof}
	We set
	\begin{equation}
		Z^{1,N}:=\frac{(-A)^{1+\alpha}\overline X^N}{\left(\mathbb{E}\int_0^T|(-A)^{1+\alpha}\overline X^N_t|_H^2\mathrm{d}t\right)^\frac{1}{2}},\quad Z^{2,N}:=\frac{(-A)^{1+\alpha}\widetilde X^N}{\left(\mathbb{E}\int_0^T|(-A)^{1+\alpha}\overline X^N_t|_H^2\mathrm{d}t\right)^\frac{1}{2}},
	\end{equation}
	and by Lemma \ref{lemFisherAsymptoticsLinear}, 
	\begin{equation}
		\mathbb{P}\left(\int_0^T|Z^{1,N}_t|_H^2\mathrm{d}t > 2\right) \leq \mathbb{P}\left(\left|\int_0^T|Z^{1,N}_t|_H^2\mathrm{d}t - 1\right|> 1\right)\rightarrow 0,
	\end{equation}
	i.e. $Z^{1,N}$ is bounded in probability. 
	Now choose $\alpha'\in (\rho^*-\frac{1}{2}, \rho^*-\frac{1}{2}+\eta)$ with $\alpha' < \alpha$. Then 
	\begin{align*}
		\int_0^T|Z^{2,N}_t|_H^2\mathrm{d}t
			&= \frac{\int_0^T|(-A)^{1+\alpha}\widetilde X^N_t|_H^2\mathrm{d}t}{\mathbb{E}\int_0^T|(-A)^{1+\alpha}\overline X^N_t|_H^2\mathrm{d}t} 
			\lesssim \frac{N^{2\gr(\alpha-\alpha')}\int_0^T|(-A)^{1+\alpha'}\widetilde X^N_t|_H^2\mathrm{d}t}{\mathbb{E}\int_0^T|(-A)^{1+\alpha}\overline X^N_t|_H^2\mathrm{d}t} \\
			& \hspace{-2cm}\lesssim \frac{N^{2\gr(\alpha-\alpha')}}{N^{\gr(2\alpha-2\rho^*+1)}}\int_0^T|(-A)^{1+\alpha'}\widetilde X^N_t|_H^2\mathrm{d}t
			= N^{-\gr(2\alpha'-2\rho^*+1)}\int_0^T|(-A)^{1+\alpha'}\widetilde X^N_t|_H^2\mathrm{d}t,
	\end{align*}
	where we used \doref{eqInverseInequality} and \doref{eqMeanCumulativeAsymptotics}. The last term converges to zero a.s. because $\int_0^T|(-A)^{1+\alpha'}\widetilde X^N_t|_H^2\mathrm{d}t<\infty$ almost surely due to condition $(R_\eta)$. Finally,
	\begin{align*}
		\mathbb{P}\left(\left|\int_0^T ( Z^{1,N}_t,Z^{2,N}_t )_H \mathrm{d}t\right|^2>\epsilon\right) 
			&\leq \mathbb{P}\left(\int_0^T|Z^{1,N}_t|_H^2\mathrm{d}t\int_0^T|Z^{2,N}_t|_H^2\mathrm{d}t>\epsilon\right) \\
			&\leq \mathbb{P}\left(\int_0^T|Z^{1,N}_t|_H^2\mathrm{d}t\geq 2\right) + \mathbb{P}\left(\int_0^T|Z^{2,N}_t|_H^2\mathrm{d}t\geq \frac{\epsilon}{2}\right),
	\end{align*}
	which converges to zero as $N\rightarrow\infty$. The claim follows easily. 
\end{proof}

\subsection{Analysis of the Estimators}
\label{subsecAnalysisEstimators}

Throughout this section we work under the assumptions of Theorem \ref{thmMain}. Inserting \doref{eqProcessGalerkin} into \doref{eqThetaone}, \doref{eqThetatwo} and \doref{eqThetathree}, 
the estimators can be written in the form 
\begin{equation}\label{eqThetaoneDecomposition}
	\thetaone - \theta = -\frac{\int_0^T\langle(-A)^{1+2\alpha}X^N_t,P_NB\mathrm{d}W_t\rangle}{\int_0^T|(-A)^{1+\alpha} X^N_t|_H^2\mathrm{d}t},
\end{equation}
\begin{equation}\label{eqThetatwoDecomposition}
	\thetatwo - \theta = -\frac{\int_0^T\langle(-A)^{1+2\alpha}X^N_t,P_NB\mathrm{d}W_t\rangle}{\int_0^T|(-A)^{1+\alpha} X^N_t|_H^2\mathrm{d}t} - \mathrm{bias}_N(X) + \mathrm{bias}_N(X^N),
\end{equation}
\begin{equation}\label{eqThetathreeDecomposition}
	\thetathree - \theta = -\frac{\int_0^T\langle(-A)^{1+2\alpha}X^N_t,P_NB\mathrm{d}W_t\rangle}{\int_0^T|(-A)^{1+\alpha} X^N_t|_H^2\mathrm{d}t} - \mathrm{bias}_N(X).
\end{equation}

We prove asymptotic normality of $\thetaone$ by means of the following CLT, which is a special case of \citep[Theorem 5.5.4 (I)]{LiptserShiryayev89} and \citep[Theorem VIII.4.17]{JacodShiryayev03}:
\begin{lemma}\label{thmCLT}
	Let $(M^N)_{N\in\N}$ be a sequence of continuous local martingales with $M^N_0 = 0$, let $T>0$ and $V>0$. Assume
	\begin{equation}
		\langle M^N\rangle_T\xrightarrow{\mathbb{P}} V
	\end{equation}
	as $N\rightarrow\infty$. Then
	\begin{equation}
		M^N_T\xrightarrow{d} \mathcal{N}(0,V).
	\end{equation}
\end{lemma}

In the present situation, 
we set
\begin{equation}
	M^N_t:=N^{-\gr(2\alpha-2\gamma)-\frac{\gr+1}{2}}\int_0^t\langle(-A)^{1+2\alpha} X^N_t,P_NB\mathrm{d}W_t\rangle
\end{equation}
for $\alpha > \gamma - \frac{1+\gr^{-1}}{8}$ and note that these are continuous local martingales with
\begin{equation}
	\langle M^N\rangle_t = N^{-\gr(4\alpha-4\gamma+1)-1}\int_0^t|(-A)^{1+2\alpha-\gamma}X^N_s|_H^2\mathrm{d}s < \infty
\end{equation}
almost surely. 
Proposition \ref{propFisherAsymptoticsNonlinear} and \doref{eqMeanCumulativeAsymptotics} give $\langle M^N\rangle_T\rightarrow \CExp_{2\alpha-\gamma}$ in probability. The CLT gives $M^N_T\xrightarrow{d}\mathcal{N}(0,\CExp_{2\alpha-\gamma})$. Another application of Proposition \ref{propFisherAsymptoticsNonlinear} together with Slutsky's lemma yields 
\begin{equation}
	N^\frac{\gr+1}{2}\frac{\int_0^T\langle(-A)^{1+2\alpha}X^N_t,P_NB\mathrm{d}W_t\rangle}{\int_0^T|(-A)^{1+\alpha}X^N_t|_H^2\mathrm{d}t}\xrightarrow{d}\mathcal{N}\left(0,\frac{\CExp_{2\alpha-\gamma}}{(\CExp_\alpha)^2}\right).
\end{equation}
Rearranging the terms, we have proven part (ii) from Theorem \ref{thmMain}. 
\begin{remark}
	It is not necessary to perform a perturbation argument to prove asymptotic normality for $\thetaone$, i.e. we do not have to bound a remainder integral of the type $\int_0^T\langle(-A)^{1+2\alpha} \widetilde X^N_t,P_NB\mathrm{d}W_t\rangle$ directly (even if this is not difficult using the Burkholder--Davis--Gundy inequality). 
\end{remark}
Next, we prove consistency for the remaining estimators. Taking into account \doref{eqThetatwoDecomposition} and \doref{eqThetathreeDecomposition}, part (i) and (iv) from Theorem \ref{thmMain} follow immediately from the following lemma:

\begin{lemma}\label{lemBiasS}
	Let $\rho\in[0,\rho^*)$. Assume $(A_\rho)$ and either $(S_\rho)$ or $(S'_\rho)$ hold 
	with $\rho+\epsilon_\rho > \rho^*$, assume further $\widetilde X\in R(\rho+\epsilon_\rho)$. 
	Let $\condalpha$. 
	\begin{enumerate}
		\item If $\rho + \epsilon_\rho - \rho^* > \frac{1 + \gr^{-1}}{2}$, then a.s. 
		\begin{equation}
			\lim_{N\rightarrow\infty}N^{\frac{\gr+1}{2}}\mathrm{bias}_N(X) = 0.
		\end{equation}
		\item Otherwise,
		\begin{equation}\label{eqBiasToZeroSecondCase}
			\lim_{N\rightarrow\infty} N^{\gr \epsilon}\mathrm{bias}_N(X) = 0
		\end{equation}
		a.s. for any $\epsilon < \rho+\epsilon_\rho-\rho^*$.
	\end{enumerate}
	The same is true for $\mathrm{bias}_N(X^N)$.
\end{lemma}
\begin{proof}
	We prove the statement just for $\mathrm{bias}_N(X)$, the proof for the remaining case is identical up to trivial norm estimates. If $\rho+\epsilon_\rho-\rho^* \leq \frac{1+\gr^{-1}}{2}$, choose $\epsilon \in (0, \rho+\epsilon_\rho-\rho^*)$, otherwise choose $\epsilon\in\left(\frac{1+\gr^{-1}}{2}, (2\alpha-2\rho^*+1)\wedge(\rho+\epsilon_\rho-\rho^*)\right)$. The latter interval is not empty due to $\condalpha$. In any case it holds that $2\alpha>2\rho^*- 1 + \epsilon$. 
	Now, with $\srhoexponent=2$ under $(S'_\rho)$ and $\srhoexponent=4$ under $(S_\rho)$, we have
	\begin{align*}
		\left|\int_0^T{}_V\langle (-A)^{1+2\alpha}X^N_t, P_NF(t,X_t)\rangle_{V^*} \mathrm{d}t\right| & \\
			&\hspace{-4cm}\leq \int_0^T|(-A)^{(2\alpha-2\rho^*+1-\epsilon) + (\rho^*+\frac{1}{2}-(\rho+\epsilon_\rho-\rho^*-\epsilon))}X^N_t|_H|(-A)^{\rho-\frac{1}{2}+\epsilon_\rho}P_NF(t,X_t)|_H \mathrm{d}t \\
			&\hspace{-4cm}\lesssim N^{\gr(2\alpha-2\rho^*+1-\epsilon)}\int_0^T|X^N_t|_{\rho^*+\frac{1}{2}-(\rho+\epsilon_\rho-\rho^*-\epsilon)}|P_NF(t,X_t)|_{\rho-\frac{1}{2}+\epsilon_\rho} \mathrm{d}t \\
			&\hspace{-4cm}\lesssim N^{\gr(2\alpha-2\rho^*+1-\epsilon)}\left(\int_0^T|X^N_t|_{\rho^*+\frac{1}{2}-(\rho+\epsilon_\rho-\rho^*-\epsilon)}^2\mathrm{d}t\right)^\frac{1}{2}\left(\int_0^T|P_NF(t, X_t)|_{\rho-\frac{1}{2}+\epsilon_\rho}^2\mathrm{d}t\right)^\frac{1}{2} \\
			&\hspace{-4cm}\lesssim N^{\gr(2\alpha-2\rho^*+1-\epsilon)}\left(\int_0^T(f_\rho(t)+|X_t|_{\rho+\frac{1}{2}}^\srhoexponent)g_\rho(|X_t|_\rho)\mathrm{d}t\right)^\frac{1}{2},
	\end{align*}
	where we used $\overline X, \widetilde X\in R(\rho^*-(\rho+\epsilon_\rho-\rho^*-\epsilon))$. Under $(S'_\rho)$, the last integral is bounded due to $X\in R(\rho)$ and $\srhoexponent=2$. Under $(S_\rho)$ we have $\srhoexponent=4$. In this case, $\overline X$ and $\widetilde X$ are continuous with values in $D((-A)^{\rho+\frac{1}{2}})$ due to Proposition \ref{propLinearRegularity} and $\epsilon_\rho\geq\frac{1}{2}$, thus
	\begin{equation*}
		\int_0^T(f_\rho(t)+|X_t|_{\rho+\frac{1}{2}}^4)g_\rho(|X_t|_\rho)\mathrm{d}t \leq \sup_{0\leq t\leq T}g_\rho(|X_t|_\rho)\left(\int_0^Tf_\rho(t)\mathrm{d}t + T\sup_{0\leq t\leq T}|X_t|_{\rho+\frac{1}{2}}^4\right) < \infty.
	\end{equation*}
	In any case, 
	\begin{equation}
		\limsup_{N\rightarrow\infty}N^{-\gr(2\alpha-2\rho^*+1-\epsilon)}\left|\int_0^T{}_V\langle (-A)^{1+2\alpha}X^N_t, P_NF(t,X_t)\rangle_{V^*} \mathrm{d}t\right| < \infty,
	\end{equation}
	and the claim follows. For \doref{eqBiasToZeroSecondCase}, we note that $\epsilon\in(0,\rho+\epsilon_\rho-\rho^*)$ is arbitrary. 
\end{proof}

Finally, Theorem \ref{thmMain} (iii) follows from \doref{eqThetatwoDecomposition} 
and the next lemma.

\begin{lemma}\label{lemBiasT}
	Assume $(A_\rho)$ and $(T_{\rho})$ for some $\rho\geq 0$, let $\condalpha$ and $\epsilon>0$ such that $X\in R(\rho^*-\epsilon)$ almost surely. Then 
	\begin{equation}
		N^{\gr\delta_\rho - 2\gr\epsilon}(\mathrm{bias}_N(X)-\mathrm{bias}_N(X^N))\xrightarrow{a.s.}0,
	\end{equation}
	where $\delta_\rho$ is as in $(T_\rho)$. 
\end{lemma}
\begin{proof}
	We proceed similarly as in Lemma \ref{lemBiasS}. Since $\alpha>\rho^*-\frac{1}{2}$, it holds 
	\begin{align*}
		\left|\int_0^T{}_V\langle (-A)^{1+2\alpha}X^N_t, P_NF(t, X_t)-P_NF(t, X^N_t)\rangle_{V^*} \mathrm{d}t\right| & & \\
			& \hspace{-8cm}\leq \int_0^T|(-A)^{\frac{3}{2}+2\alpha-\rho}X^N_t|_H|(-A)^{\rho-\frac{1}{2}}(P_NF(t, X_t)-P_NF(t, X^N_t))|_H \mathrm{d}t \\
			& \hspace{-8cm}\lesssim N^{\gr(2\alpha-\rho-\rho^*+\epsilon+1)}\int_0^T|X^N_t|_{\rho^*+\frac{1}{2}-\epsilon}|P_NF(t, X_t)-P_NF(t, X^N_t)|_{\rho-\frac{1}{2}} \mathrm{d}t \\
			& \hspace{-8cm}\leq N^{\gr(2\alpha-\rho-\rho^*+\epsilon+1)}\left(\int_0^T|X^N_t|_{\rho^*+\frac{1}{2}-\epsilon}^2\mathrm{d}t\int_0^T|P_NF(t, X_t)-P_NF(t, X^N_t)|_{\rho-\frac{1}{2}}^2 \mathrm{d}t\right)^\frac{1}{2} \\
			& \hspace{-8cm}\lesssim N^{\gr(2\alpha-\rho-\rho^*+\epsilon+1)}\left(\int_0^Th_\rho(|X_t|_\rho, |X^N_t|_\rho)| X_t- X^N_t|_{\rho+\frac{1}{2}-\delta_\rho}^{2}\mathrm{d}t\right)^\frac{1}{2} \\
			& \hspace{-8cm}\lesssim N^{\gr(2\alpha-\rho-\rho^*+\epsilon+1)-\gr(\delta_\rho+\rho^*-\epsilon-\rho)}\left(\int_0^T| X_t- X^N_t|_{\rho^*+\frac{1}{2}-\epsilon}^{2}\mathrm{d}t\right)^\frac{1}{2},
	\end{align*}
	where we used \doref{eqDirectInequality} and the fact that $h_\rho$ is bounded on $[0,\;\sup_{0\leq t\leq T}|X_t|_\rho]^2$.
	Thus
	\begin{align*}
		\limsup_{N\rightarrow\infty}N^{-\gr(2\alpha-2\rho^*+1)+(\gr\delta_\rho-2\gr\epsilon)}\left|\int_0^T{}_V\langle (-A)^{1+2\alpha}X^N_t, P_NF(t, X_t)-P_NF(t, X^N_t)\rangle_{V^*} \mathrm{d}t\right| & & \\
			& \hspace{-6cm}\leq C \limsup_{N\rightarrow\infty}\left(\int_0^T| X_t- X^N_t|_{\rho^*+\frac{1}{2}-\epsilon}^{2}\mathrm{d}t\right)^\frac{1}{2} = 0
	\end{align*}
	a.s. for some $C>0$ by dominated convergence.
\end{proof}

\section{The Case of Coupled SPDEs}
\label{secExtended}

The same techniques as applied above allow for further generalization. More precisely, $X$ may be coupled with another state variable $X^\perp$ with state space $H^\perp$. This leads to a system of the form 
\begin{equation}\label{eqCoupledTwoEquations}
\begin{matrix}
	\mathrm{d}X_t &=& (\theta A X_t &+ &F(t,X_t, X^\perp_t))\mathrm{d}t &+& B\mathrm{d}W_t, \\
	\mathrm{d}X^\perp_t &=& & &F^\perp(t, X_t, X^\perp_t)\mathrm{d}t &+& B^\perp(t,X_t, X^\perp_t)\mathrm{d}W_t 
\end{matrix}
\end{equation}
with initial condition $X_0\in H$, $X^\perp_0\in H^\perp$. \\

Let us describe this setting in more detail. 
Let $\mathscr{H}$ be a Hilbert space with inner product $(\cdot,\cdot)_{\mathscr{H}}$ and $H\subseteq\mathscr{H}$ a closed subspace with orthogonal complement $H^\perp$, i.e. $\mathscr{H} = H\oplus H^\perp$. Let $A$ be some negative definite self-adjoint operator on $H$ with compact resolvent and domain $D(A)\subset H$, let $\mathscr{A} = A\oplus 0$ be its trivial continuation to $D(\mathscr{A}) = D(A)\oplus H^\perp\subset\mathscr{H}$ given by $\mathscr{A}(h,h^\perp) = (Ah,0)$. We set $V=D((-A)^\frac{1}{2})$ and $\mathscr V = D((-\mathscr A)^\frac{1}{2}) = V\oplus H^\perp$. Consider an equation in $\mathscr H$ of the form
\begin{equation}\label{eqCompleteScrHSecond}
	\mathrm{d}\mathscr{X}_t = (\theta \mathscr{A}\mathscr{X}_t + \mathscr{F}(t,\mathscr{X}_t))\mathrm{d}t + \mathscr{B}(t,\mathscr{X}_t)\mathrm{d}W_t
\end{equation}
with initial condition $\mathscr{X}_0\in\mathscr{H}$. Here, 
$\mathscr F:[0,T]\times \mathscr V\rightarrow \mathscr V^*$ is a measurable operator, 
$W$ is a cylindrical Wiener process on $\mathscr H$, 
and $\mathscr B: [0,T]\times\mathscr V\times\Omega\rightarrow L_2(\mathscr H)$ is 
measurable with values in the space of Hilbert--Schmidt operators on $\mathscr{H}$. 
By decomposing $\mathscr X$, $\mathscr F$ and $\mathscr B$ as $\mathscr X = (X,X^\perp)$, $\mathscr F = (F,F^\perp)$ and $\mathscr B = (B,B^\perp)$, we obtain \doref{eqCoupledTwoEquations}. \\

As before, we assume $B = (-A)^{-\gamma}$, whereas $B^\perp:[0,T]\times\mathscr V\times\Omega \rightarrow H^\perp$ may be arbitrary. 
The eigenvalues $(\lambda_k)_{k\in\N}$ of $-A$ are assumed to satisfy \doref{eqLambdaAsymp}. The Sobolev norms on the spaces $D((-A)^\rho)\subset H$ and $D((-\mathscr A)^\rho) \subset\mathscr H$ are given by $|x|_\rho = |(-A)^\rho x|_{H}$ and $||x||_\rho = |(-\mathscr A)^\rho x|_{\mathscr H}$, respectively. It is easy to verify that $D((-\mathscr A)^\rho) = D((-A)^\rho)\oplus H^\perp$ for $\rho\in\R$. We define
\begin{equation}
	\mathscr R(\rho) := C([0,T]; D((-\mathscr A)^\rho))\cap L^2([0,T]; D((-\mathscr A)^{\rho+\frac{1}{2}}))
\end{equation}
and
\begin{equation}
	R(\rho) := C([0,T]; D((-A)^\rho))\cap L^2([0,T]; D((-A)^{\rho+\frac{1}{2}}))
\end{equation}
and say that \doref{eqCompleteScrHSecond} has a 
{\it weak solution} in $\mathscr R(\rho)$ on $[0,T]$ if there is a stochastic basis $(\Omega, \mathcal{F}, (\mathcal{F})_{t\geq 0}, \mathbb{P})$, a cylindrical Wiener process $W$ on $\mathscr H$ and some $(\mathcal{F}_t)_{t\geq 0}$-adapted process $\mathscr X\in \mathscr R(\rho)$ which fulfils a.s.
\begin{equation}\label{eqWeakSolutionExtended}
	\mathscr X_t = \mathscr X_0 + \int_0^t\left(\theta \mathscr A \mathscr X_s + \mathscr F(s, \mathscr X_s)\right)\mathrm{d}s + \int_0^t\mathscr B(s)\mathrm{d}W_s
\end{equation}
for $t\in [0,T]$. Condition $(A_\rho)$ can be adapted to the new setting:
\begin{enumerate}
	\item[$(\mathscr A_\rho)$] 
		The process $\mathscr X$ is a unique (in the sense of probability law) weak solution to (\ref{eqCompleteScrHSecond}) 
		on $[0,T]$ with $\mathscr X\in \mathscr R(\rho)$ a.s.\footnote{As before, this means that a stochastic basis and a cylindrical Wiener process can be found such that \doref{eqWeakSolutionExtended} is satisfied.}
\end{enumerate}
If $(\mathscr A_0)$ holds, then higher regularity $(\mathscr A_\rho)$, $\rho > 0$, is equivalent to $X\in R(\rho)$ almost surely. The conditions $(S'_\rho)$ and $(T_\rho)$ have the following modified counterparts:
\begin{enumerate}
	\item[$(\mathscr S'_\rho)$] There is $\epsilon_\rho>0$, an integrable function $f_\rho\in L^1(0,T;\R)$ and a continuous function $g_\rho:[0,\infty)\rightarrow[0,\infty)$ such that 
		\begin{equation}
			|F(t,v)|_{\rho - \frac{1}{2}+\epsilon_\rho}^2 \leq (f_{\rho}(t) + ||v||_{\rho+\frac{1}{2}}^2)g_\rho(||v||_{\rho})
		\end{equation}
		for any $t\in [0,T]$ and $v\in D((-\mathscr A)^{\rho+\frac{1}{2}})$. 
\end{enumerate}
\begin{enumerate}
	\item[$(\mathscr T_\rho)$] There is $\delta_\rho>0$ 
	and a continuous function $h_\rho:[0,\infty)^2\rightarrow[0,\infty)$ such that 
		\begin{equation}
			|F(t,u) - F(t,v)|_{\rho - \frac{1}{2}}^2 \leq h_\rho(||u||_\rho,||v||_\rho)||u-v||_{\rho+\frac{1}{2}-\delta_\rho}^{2}
		\end{equation}
		for $t\in[0,T]$ and $u,v\in D((-\mathscr A)^{\rho+\frac{1}{2}})$.\end{enumerate}

In analogy to Section \ref{secStatisticalInference}, we can construct four estimators for $\theta$ as follows:

\begin{enumerate}
	\item The first approach reads as
\begin{equation}
	\thetaone := - \frac{\int_0^T\langle (-A)^{1+2\alpha}X^N_t,\mathrm{d}X^N_t\rangle}{\int_0^T|(-A)^{1+\alpha}X^N_t|_H^2\mathrm{d}t} + \mathrm{bias}_N(\mathscr X),
\end{equation}
where
\begin{equation}
\mathrm{bias}_N(\mathscr U) := \frac{\int_0^T{}_V\langle (-A)^{1+2\alpha}X^N_t, P_NF(t,\mathscr U_t)\rangle_{V^*} \mathrm{d}t}{\int_0^T|(-A)^{1+\alpha} X^N_t|_H^2\mathrm{d}t}.
\end{equation}
If continuous-time observation of the full solution $(\mathscr X_t)_{t\in[0,T]}$ is given, this is a feasible estimator. 
\item A second possibility is that we observe just $(\mathscr X^N_t)_{t\in[0,T]}$, where $\mathscr X^N = (X^N,X^\perp)$ and $X^N=P_NX$. In this case, we can adapt the bias term:
\begin{equation}
	\thetatwoone := - \frac{\int_0^T\langle (-A)^{1+2\alpha}X^N_t,\mathrm{d}X^N_t\rangle}{\int_0^T|(-A)^{1+\alpha}X^N_t|_H^2\mathrm{d}t} + \mathrm{bias}_N(\mathscr X^N).
\end{equation}
\item The observation scheme may be even more restrictive in the sense that just $(X^N_t)_{t\in[0,T]}$ is observed without any knowledge of $X^\perp$. 
In this case the natural estimator is
\begin{equation}
	\thetatwotwo := - \frac{\int_0^T\langle (-A)^{1+2\alpha}X^N_t,\mathrm{d}X^N_t\rangle}{\int_0^T|(-A)^{1+\alpha}X^N_t|_H^2\mathrm{d}t} + \mathrm{bias}_N(X^N),
\end{equation}
where we identified the $H$-valued process $X^N$ with its trivial extension $(X^N,0)$ to $\mathscr H$. 
\item 
Finally, we can drop the nonlinear term completely:
\begin{equation}
	\thetathree := - \frac{\int_0^T\langle (-A)^{1+2\alpha}X^N_t,\mathrm{d}X^N_t\rangle}{\int_0^T|(-A)^{1+\alpha}X^N_t|_H^2\mathrm{d}t}.
\end{equation}
This estimator uses information that is accessible in any of the preceding observation schemes. 
\end{enumerate}
The proof of Theorem \ref{thmMain} gives immediately the following extension: 

\begin{theorem}
	Assume $(\mathscr A_\rho)$ and $(\mathscr S'_\rho)$ hold for $\rho\in[0,\rho^*)$ such that $\rho+\epsilon_\rho > \rho^*$. 
	Let $\condalpha$. 
	\begin{enumerate}
		\item All estimators $\thetaone, \thetatwoone, \thetatwotwo, \thetathree$ are consistent as $N\rightarrow\infty$. 
		\item $\thetaone$ is asymptotically normal. More precisely, 
		\begin{equation}\label{eqAsymptoticNormalitySchemeSecond}
			N^\frac{\gr+1}{2}(\thetaone-\theta)\rightarrow\mathcal{N}\left(0,\frac{2\theta (\gr(2\alpha-2\gamma+1)+1)^2}{T\Lambda^{2\alpha-2\gamma+1}(\gr(4\alpha-4\gamma+1)+1)}\right)
		\end{equation}
		in distribution as $N\rightarrow\infty$. 
		\item If $(\mathscr T_\rho)$ holds for some $\rho\in[0,\rho^*)$ with $\delta_\rho > \frac{1 + \gr^{-1}}{2}$, then $\thetatwoone$ is asymptotically normal as in (\ref{eqAsymptoticNormalitySchemeSecond}). Otherwise $N^a(\thetatwoone-\theta)\xrightarrow{\mathbb{P}}0$ for each $a<\gr\delta_\rho$. 
		\item If $\rho+\epsilon_\rho-\rho^* > \frac{1+\gr^{-1}}{2}$, where $\epsilon_\rho$ is as in $(\mathscr S'_\rho)$, then $\thetatwoone$, $\thetatwotwo$ and $\thetathree$ are asymptotically normal as in (\ref{eqAsymptoticNormalitySchemeSecond}). Otherwise, $N^a(\thetatwoone - \theta)\xrightarrow{\mathbb{P}}0$ for each $a < \gr(\rho+\epsilon_\rho-\rho^*)$, and the same is true for $\thetatwotwo$ and $\thetathree$. 
	\end{enumerate}
\end{theorem}

Note that Lemma \ref{lemBiasT} does not transfer to $\thetatwotwo$ without further assumptions. The reason is that $||\mathscr X - X^N||_{\rho+\frac{1}{2}-\delta_\rho}=|X-X^N|_{\rho+\frac{1}{2}-\delta_\rho} + |X^\perp|_{H^\perp}$, where the second summand cannot be controlled as $N\rightarrow\infty$. \\

\begin{example}
As an illustration for the theory developed in this section, consider a stochastic Fitzhugh--Nagumo system (\citep{Fitzhugh61, Nagumo62}) 
of the type
\begin{align*}
	\mathrm{d}v_t &= (\theta\Delta v_t + v_t(1-v_t)(v_t-a) - w_t)\mathrm{d}t + \sigma(-\Delta)^{-\gamma}\mathrm{d}W^{(1)}_t, \\
	\mathrm{d}w_t &= \epsilon(v_t-bw_t)\mathrm{d}t + B^{\perp}(t, v_t, w_t)\mathrm{d}W^{(2)}_t
\end{align*}
on a bounded interval $I\subset \R$ with Neumann boundary conditions, where $a\in (0,1)$, $b\geq 0$ and $\epsilon, \sigma > 0$ are constants. Models of that type are well-studied, e.g. in neuroscience. Note that the Laplacian is contained only in the drift term of the first variable. The nonlinearity $F(v, w)$ is cubic in $v$. Computations similar to Proposition \ref{propReactionDiffusion} 
show that $(\mathscr S'_\rho)$ holds for any $\rho\geq 0$ with $\epsilon_\rho = \frac{1}{2}+\frac{1}{3} = \frac{5}{6}$. Consequently, $\thetaone$ is asymptotically normal. Similarly, $(\mathscr T_\rho)$ holds for $\rho>\frac{1}{4}+\frac{1}{2} = \frac{3}{4}$ with $\delta_\rho = 1$, so $\thetatwoone$ is asymptotically normal if $v\in R(\rho)$, $\rho>\frac{3}{4}$. \\

However, in many applications it would be even more natural to drop the noise $W^{(1)}$ from the equation for $v_t$, i.e. to set $\sigma=0$. In this case, the linearization of the equation for $v_t$ reduces to the heat equation with analytic solution, so that the perturbation argument used throughout this work does not apply. New methods have to be developed for this situation. 
\end{example}

\begin{appendix}

\section{Well-Posedness of a Class of Semilinear Evolution Equations}\label{appWellPosed}

The purpose of this section is to provide a short and self-contained study on the well-posedness of 
\begin{equation}\label{eqCompleteAppendix}
	\mathrm{d}X_t = (\theta AX_t + F(t, X_t))\mathrm{d}t + (-A)^{-\gamma}\mathrm{d}W_t.
\end{equation}
This problem is well understood, and there is a vast literature on this topic, see e.g. \citep{DaPratoZabczyk14, LiuRockner15} and references therein.\footnote{We point out \citep[Section 7.2]{DaPratoZabczyk14}, where a coercivity condition similar to $(C_\rho)$ is used, and \citep{Peszat95} for the existence and uniqueness of mild solutions to semilinear stochastic equations in suitably regular Banach spaces. See \citep{Cerrai01} for a detailed analysis of semilinear equations with Nemytskii-type nonlinearities.} Still, to the best of our knowledge, there are few results that deal explicitly with higher regularity for the nonlinear part of solutions to \doref{eqCompleteAppendix}. 
However, this type of regularity result is needed for the statistical analysis we conduct in this work. We aim at a concise presentation rather than a general framework. 
Remember that the regularity limit $\rho^*$ is given by
\begin{equation*}
	\rho^* = \gamma - \frac{\gr^{-1}}{2},
\end{equation*}
where $\gr$ comes from \doref{eqLambdaAsymp}. Fix a stochastic basis $(\Omega, \mathcal{F}, (\mathcal{F}_t)_{t\geq 0}, \mathbb{P})$ and a cylindrical Wiener process $W$. By {\it strong uniqueness} in $R(\rho)$ we mean that two solutions $X, Y\in R(\rho)$ to \doref{eqCompleteAppendix} with $X_0=Y_0$ a.s. satisfy $X_t=Y_t$ for all $t\in[0,T]$ a.s.
By Proposition \ref{propLinearRegularity} there is a solution $\overline X$ to \doref{eqCompleteAppendix} with $F=0$ with 
\begin{equation}\label{eqRegularityOfLinearPart}
	\overline X\in C([0,T];D((-A)^{\rho+\frac{1}{2}}))
\end{equation}
for any $0\leq\rho<\rho^*$. 
The process $\widetilde X := X - \overline X$ will be constructed to be a solution to 
\begin{equation}\label{eqNonlinearProcess}
	\mathrm{d}\widetilde X_t = (\theta A \widetilde X_t + F(t, \overline X_t + \widetilde X_t))\mathrm{d}t,\quad \widetilde X_0 = X_0.
\end{equation}

\begin{theorem}\label{thmAppendixWellPosed}
	Let $(S_\rho)$ hold for any $0\leq\rho<\rho^*$ and $(C_{\rho_1})$, $(T_{\rho_2})$ for some $\rho_1, \rho_2<\rho^*$ such that $\delta_{\rho_2}\geq\frac{1}{2}$. Then there is a strongly unique solution $X$ to \doref{eqCompleteAppendix} in $R(\rho)$ for any $0\leq\rho<\rho^*$ a.s. Furthermore, $\widetilde X\in R(\rho^*+\eta)$ for every $\eta < \sup_{0\leq\rho<\rho^*}(\rho+\epsilon_\rho-\rho^*)$. 
\end{theorem}

For the forthcoming calculations, we fix a realization $\omega\in\Omega$ from a suitable set of probability one. We define Galerkin approximations for $\widetilde X$:
\begin{lemma}\label{lemExistenceGalerkin}
	Under assumption $(C_{\rho})$ for some $\rho < \rho^*$ there is a continuous solution $\widetilde X^N$ in $P_NH\simeq\R^N$ on $[0,T]$ to 
	\begin{equation}\label{eqNonlinearGalerkin}
		\mathrm{d}\widetilde X^N_t = (\theta A\widetilde X^N_t + P_N F(t, \overline X_t + \widetilde X^N_t))\mathrm{d}t
	\end{equation}
	with $\widetilde X^N_0=X^N_0$. Furthermore, $(\widetilde X^N)_{N\in\N}$ is bounded in $R(0)$. 
\end{lemma}
\begin{proof}
	By assumption, $F_N:[0,T]\times P_NH\rightarrow P_NH,\;(t,x)\mapsto P_NF(t, \overline X_t + x)$ is continuous. 
	The Peano existence theorem yields a local solution up to some time $T_0$. Now,
	\begin{align*}
		|\widetilde X^N_t|_H^2 = |X^N_0|_H^2 + 2\int_0^t{}_{V^*}\langle \theta A\widetilde X^N_s + P_NF(s, \overline X_s + \widetilde X^N_s),\widetilde X^N_s\rangle_V\mathrm{d}s,
	\end{align*}
	so by $(C_{\rho})$, we get
	\begin{align*}
		|\widetilde X^N_t|_H^2 + 2\theta\int_0^t|\widetilde X^N_s|_V^2\mathrm{d}s & = |X^N_0|_H^2 + 2\int_0^t{}_{V^*}\langle F(s, \overline X_s + \widetilde X^N_s),\widetilde X^N_s\rangle_V\mathrm{d}s \\
		& \leq |X^N_0|_H^2 + 2\int_0^t(1+|\widetilde X^N_s|_H^2)b_{\rho}(|\overline X_s|_{\rho+\frac{1}{2}})\mathrm{d}s,
	\end{align*}
	and by Gronwall's inequality and \doref{eqRegularityOfLinearPart},
	\begin{align*}
		|\widetilde X^N_t|_H^2 \leq \left(|X_0|_H^2+B_T\right)\left(1+B_Te^{B_T}\right)\leq C_T < \infty,
	\end{align*}
	where $B_T := 2T\sup_{0\leq s\leq T}b_\rho(|\overline X_s|_{\rho+\frac{1}{2}})$. Thus $\widetilde X^N_t$ exists up to time $T$, and
	\begin{align*}
		\sup_{0\leq t\leq T}|\widetilde X^N_t|_H^2 + 2\theta\int_0^T|\widetilde X^N_s|_V^2\mathrm{d}s\leq |X_0|_H^2 + 2T(1+C_T^2)\sup_{0\leq t\leq T}b_\rho(|\overline X_t|_{\rho+\frac{1}{2}}) < \infty.
	\end{align*}
\end{proof}

\begin{lemma}\label{lemBoundsGalerkin}
	Let $\rho<\rho^*$. If $(\widetilde X^N)_{N\in\N}$ is bounded in $R(\rho)$ and $(S_\rho)$ holds, then the following is true:
	\begin{enumerate}
		\item $(\widetilde X^N)_{N\in\N}$ is bounded in $R(\rho+\epsilon_\rho)$. 
		\item $(\frac{\mathrm{d}}{\mathrm{d}t}\widetilde X^N)_{N\in\N}$ is bounded in $L^2([0,T], D((-A)^{\rho-\frac{1}{2}+\epsilon_\rho}))$.
		\item $(\widetilde X^N)_{N\in\N}$ has a subsequence converging strongly in $L^2([0,T],D((-A)^{\rho+\epsilon_\rho}))$.
	\end{enumerate}
\end{lemma}
\begin{proof}
As before, we have 
\begin{equation*}
	|\widetilde X^N_t|_{\rho+\epsilon_\rho}^2 = |X^N_0|_{\rho+\epsilon_\rho}^2 + 2\int_0^t{}_{V^*}\langle(-A)^{\rho+\epsilon_\rho}(\theta A \widetilde X^N_s + P_NF(s, \overline X_s + \widetilde X^N_s)), (-A)^{\rho+\epsilon_\rho}\widetilde X^N_s\rangle_V\mathrm{d}s,
\end{equation*}
thus
\begin{align*}
	|\widetilde X^N_t|_{\rho+\epsilon_\rho}^2 + 2\theta\int_0^t|\widetilde X^N_s|_{\rho+\frac{1}{2}+\epsilon_\rho}^2\mathrm{d}s & \\
	& \hspace{-3cm} = |X^N_0|_{\rho+\epsilon_\rho}^2 + 2\int_0^t{}_{V^*}\langle(-A)^{\rho+\epsilon_\rho}P_NF(s, \overline X_s + \widetilde X^N_s), (-A)^{\rho+\epsilon_\rho}\widetilde X^N_s\rangle_V\mathrm{d}s.
\end{align*}
We obtain 
\begin{align*}
	& \sup_{t\in [0,T]}|\widetilde X^N_t|_{\rho+\epsilon_\rho}^2 + 2\theta\int_0^{T}|\widetilde X^N_s|_{\rho+\frac{1}{2}+\epsilon_\rho}^2\mathrm{d}s \\
		&\hspace{1cm}\leq |X^N_0|_{\rho+\epsilon_\rho}^2 + 2 \int_0^{T} \left|{}_{V^*}\langle(-A)^{\rho+\epsilon_\rho}P_NF(s, \overline X_s + \widetilde X^N_s), (-A)^{\rho+\epsilon_\rho}\widetilde X^N_s\rangle_V\right|\mathrm{d}s \\
		&\hspace{1cm}\leq |X^N_0|_{\rho+\epsilon_\rho}^2 + 2 \int_0^{T} |P_NF(s, \overline X_s + \widetilde X^N_s)|_{\rho-\frac{1}{2}+\epsilon_\rho} |\widetilde X^N_s|_{\rho+\frac{1}{2}+\epsilon_\rho}\mathrm{d}s \\
		&\hspace{1cm}\leq |X^N_0|_{\rho+\epsilon_\rho}^2 + C_\theta \int_0^{T} |P_NF(s, \overline X_s + \widetilde X^N_s)|_{\rho-\frac{1}{2}+\epsilon_\rho}^2\mathrm{d}s + \theta\int_0^{T}|\widetilde X^N_s|_{\rho+\frac{1}{2}+\epsilon_\rho}^2\mathrm{d}s,
\end{align*}
where we made use of Young's inequality in the last step. Thus
\begin{equation*}
	\sup_{t\in [0,T]}|\widetilde X^N_t|_{\rho+\epsilon_\rho}^2 + \theta\int_0^{T}|\widetilde X^N_s|_{\rho+\frac{1}{2}+\epsilon_\rho}^2\mathrm{d}s \leq |X^N_0|_{\rho+\epsilon_\rho}^2 + C_\theta \int_0^{T} |P_NF(s, \overline X_s + \widetilde X^N_s)|_{\rho-\frac{1}{2}+\epsilon_\rho}^2\mathrm{d}s.
\end{equation*}
Using $(S_\rho)$, 
we obtain that
\begin{align*}
	\int_0^{T} |P_NF(t, \overline X_t + \widetilde X^N_t)|_{\rho-\frac{1}{2}+\epsilon_\rho}^2\mathrm{d}t
		\leq & \int_0^T(f_{\rho}(t) + |\overline X_s + \widetilde X^N_s|_{\rho+\frac{1}{2}}^4)g_\rho(|\overline X_s + \widetilde X^N_s|_{\rho})\mathrm{d}t \\
		& \hspace{-4,5cm} \leq \sup_{0\leq t\leq T}g_\rho\left(|\overline X_t+\widetilde X^N_t|_{\rho}\right)\left(\int_0^Tf_{\rho}(t)\mathrm{d}t + 2\sup_{0\leq t\leq T}|\overline X_t|_{\rho+\frac{1}{2}}^2\int_0^T|\overline X_t+\widetilde X^N_t|_{\rho+\frac{1}{2}}^2\mathrm{d}t\right) \\
		& \hspace{-4cm} + 2\sup_{0\leq t\leq T}g_\rho(|\overline X_t+\widetilde X^N_t|_\rho)\int_0^T|\overline X_t+\widetilde X^N_t|_{\rho+\frac{1}{2}}^2\sup_{0\leq s\leq t}|\widetilde X^N_s|_{\rho+\frac{1}{2}}^2\mathrm{d}t,
\end{align*}
in particular, as $\epsilon_\rho\geq\frac{1}{2}$:
\begin{equation}\label{eqWidetildeXGronwall}
	\sup_{0\leq t\leq T}|\widetilde X^N_t|_{\rho+\epsilon_\rho}^2 + \theta\int_0^T|\widetilde X^N_s|_{\rho+\frac{1}{2}+\epsilon_\rho}^2\mathrm{d}s\leq a_T + b_T\int_0^T|\overline X_t + \widetilde X^N_t|_{\rho+\frac{1}{2}}^2\sup_{0\leq s\leq t}|\widetilde X^N_s|_{\rho+\epsilon_\rho}^2\mathrm{d}t,
\end{equation}
where $a_T$ and $b_T$ are finite due to \doref{eqRegularityOfLinearPart} 
and the assumption that $(\widetilde X^N)_{N\in\N}$ is bounded in $R(\rho)$. 
Now, an application of Gronwall's lemma gives that the left-hand side of \doref{eqWidetildeXGronwall} is bounded uniformly in $N$, i.e. $(\widetilde X^N)_{N\in\N}$ is bounded in $R(\rho+\epsilon_\rho)$. \\

With respect to (ii), note that $\frac{\mathrm{d}}{\mathrm{d}t}\widetilde X^N_t = \theta A\widetilde X^N_t + P_NF(t, \overline X_t + \widetilde X^N_t)$ $\mathrm{d}t$-a.e., so we apply $(S_\rho)$ one more time to get
\begin{align*}
	\int_0^T|\theta A\widetilde X^N_s + P_NF(s, \overline X_s + \widetilde X^N_s)|_{\rho-\frac{1}{2}+\epsilon_\rho}^2\mathrm{d}s & \\
	& \hspace{-4cm}\leq 2\int_0^T\left(|\widetilde X^N_s|_{\rho+\frac{1}{2}+\epsilon_\rho}^2 + (f_\rho(s) + |\overline X_s + \widetilde X^N_s|_{\rho+\frac{1}{2}}^4)g_\rho(|\overline X_s + \widetilde X^N_s|_\rho)\right)\mathrm{d}s.
\end{align*}
The right-hand side is bounded uniformly in $N$ since $\sup_{N\in\N}\sup_{0\leq t\leq T}|\widetilde X^N_t|_{\rho+\frac{1}{2}}<\infty$ due to part (i). 
Using that $D((-A)^{\rho+\frac{1}{2}+\epsilon_\rho})$ embeds compactly into $D((-A)^{\rho+\epsilon_\rho})$, 
part (iii) is now classical, see e.g. \citep[Lemma 8.4]{ConstantinFoias88}.
\end{proof}

\begin{lemma}
	Assume $(C_{\rho_1})$ and $(T_{\rho_2})$ for some $\rho_1, \rho_2 < \rho^*$ and $(S_\rho)$ for $0\leq\rho < \rho^*$. Then there is a solution $\widetilde X$ to \doref{eqNonlinearProcess} with $\widetilde X\in R(\rho^*+\eta)$ for every $\eta < \sup_{0 \leq \rho < \rho^*}(\rho+\epsilon_\rho-\rho^*)$. 
\end{lemma}
\begin{proof}
	By Lemma \ref{lemExistenceGalerkin} and Lemma \ref{lemBoundsGalerkin} assume w.l.o.g. that $(\widetilde X^N)_{N\in\N}$ is bounded in $R(\rho_2)$ and converges to some limit $\widetilde X$ strongly in $L^2([0,T],D((-A)^{\rho_2+\frac{1}{2}}))$. By $(T_{\rho_2})$,
	\begin{align*}
		\int_0^T|F(t, \overline X_t + \widetilde X_t) - F(t, \overline X_t + \widetilde X^N_t)|_{\rho_2-\frac{1}{2}}^2\mathrm{d}t & \\
		& \hspace{-3cm} \leq \sup_{0\leq t\leq T}h_{\rho_2}(|\overline X_t+\widetilde X_t|_{\rho_2}, |\overline X_t+\widetilde X^N_t|_{\rho_2})\int_0^T|\widetilde X_t-\widetilde X^N_t|_{\rho_2+\frac{1}{2}-\delta_\rho}^2\mathrm{d}t,
	\end{align*}
	so $F(t, \overline X_t + \widetilde X^N_t)$ converges to $F(t, \overline X_t + \widetilde X_t)$ strongly in $L^2([0,T],D((-A)^{\rho_2-\frac{1}{2}}))$. 
	A simple argument shows that $P_NF(t, \overline X_t + \widetilde X^N_t)$ converges to $F(t, \overline X_t + \widetilde X_t)$, too. 
	Therefore, the terms in the equation 
	\begin{equation*}
		\widetilde X^N_t = X^N_0 + \int_0^t(\theta A\widetilde X^N_s + P_NF(s, \overline X_s + \widetilde X^N_s))\mathrm{d}s
	\end{equation*}
	converge strongly in $D((-A)^{\rho_2-\frac{1}{2}})$ to their counterparts from \doref{eqNonlinearProcess} for almost every $t\in [0,T]$. It is a standard fact \citep[Theorem 3.1]{LionsMagenes72} that $\widetilde X$ has a representative in $C([0,T],D((-A)^{\rho_2}))$. The higher regularity of $\widetilde X$ follows again from Lemma \ref{lemBoundsGalerkin}. 
\end{proof}
In general, $F$ does not commute with $P_N$, so $\widetilde X^N$ cannot be identified with $P_N\widetilde X$. 
Note that in our setting the regularity of $\widetilde X$ exceeds the regularity of $\overline X$ by far. 
\begin{lemma}
	If $(T_\rho)$ holds for some $\rho\in[0,\rho^*)$ with $\delta_\rho\geq\frac{1}{2}$, then strong uniqueness holds for \doref{eqCompleteAppendix} in $R(\rho)$. 
\end{lemma}
\begin{proof}
	Let $X,Y\in R(\rho)$ be solutions to \doref{eqCompleteAppendix} with $X_0=Y_0$ a.s. As before, let $\overline X$ be the solution to \doref{eqCompleteAppendix} with $F=0$. It suffices to show that $\widetilde X_t=\widetilde Y_t$ for all $t\in[0,T]$ a.s., where $\widetilde X := X - \overline X$ and $\widetilde Y := Y - \overline X$. Both processes satisfy \doref{eqNonlinearProcess}. Thus
	\begin{align*}
		|\widetilde X_t-\widetilde Y_t|_\rho^2 = 2\int_0^t{}_{D((-A)^{\rho-\frac{1}{2}})}\langle \theta A(\widetilde X_s-\widetilde Y_s) + F(s, \overline X_s+\widetilde X_s) - F(s, \overline X_s + \widetilde Y_s),\widetilde X_s-\widetilde Y_s\rangle_{D((-A)^{\rho+\frac{1}{2}})}\mathrm{d}s,
	\end{align*}
	and Young's inequality easily gives
	\begin{align*}
		\sup_{0\leq t\leq T}|\widetilde X_t-\widetilde Y_t|_{\rho}^2 + \theta\int_0^T|\widetilde X_t-\widetilde Y_t|_{\rho+\frac{1}{2}}^2\mathrm{d}t&\leq C_\theta\int_0^T|F(t,\overline X_t+\widetilde X_t)-F(\overline X_t+\widetilde Y_t)|_{\rho-\frac{1}{2}}^2\mathrm{d}t \\
		&\hspace{-2cm}\lesssim \sup_{0\leq t\leq T}h_{\rho}(|\overline X_t+\widetilde X_t|_{\rho}^2,|\overline X_t+\widetilde Y_t|_{\rho}^2)\int_0^T\sup_{0\leq s\leq t}|\widetilde X_s-\widetilde Y_s|_{\rho}^2\mathrm{d}t,
	\end{align*}
	where we used $\delta_\rho\geq\frac{1}{2}$. Gronwall's lemma implies $\widetilde X_t=\widetilde Y_t$ for all $t\in[0,T]$. 
\end{proof}
This proves Theorem \ref{thmAppendixWellPosed}.

\end{appendix}

\section*{Acknowledgment}

This research has been 
partially
funded by Deutsche Forschungsgemeinschaft (DFG) through grant CRC 1294 "Data Assimilation", Project A01 "Statistics for Stochastic Partial Differential Equations". The authors like to thank the referees for their valuable feedback, which helped to improve the present work.

\bibliographystyle{amsplain}	
\bibliography{Literatur}

\providecommand{\bysame}{\leavevmode\hbox to3em{\hrulefill}\thinspace}
\providecommand{\MR}{\relax\ifhmode\unskip\space\fi MR }
\providecommand{\MRhref}[2]{%
  \href{http://www.ams.org/mathscinet-getitem?mr=#1}{#2}
}
\providecommand{\href}[2]{#2}
\begin{thebibliography}{10}

\bibitem{AdamsFournier03}
R.~A. Adams and J.~J.~F. Fournier, \emph{{Sobolev Spaces}}, second ed., Pure
  and Applied Mathematics, vol. 140, Elsevier/Academic Press, 2003.

\bibitem{AltmeyerReiss19}
R.~Altmeyer and M.~Rei{\ss}, \emph{{Nonparametric Estimation for Linear SPDEs
  from Local Measurements}}, {Preprint: arXiv:1903.06984 [math.ST]}, 2019.

\bibitem{BibingerTrabs17}
M.~Bibinger and M.~Trabs, \emph{{Volatility Estimation for Stochastic PDEs
  Using High-Frequency Observations}}, {Preprint: arXiv:1710.03519 [math.ST]},
  2017.

\bibitem{Cerrai01}
S.~Cerrai, \emph{{Second Order PDE's in Finite and Infinite Dimension (A
  Probabilistic Approach)}}, Lecture Notes in Mathematics, vol. 1762,
  Springer-Verlag, Berlin, 2001.

\bibitem{Chong18}
C.~Chong, \emph{{High-Frequency Analysis of Parabolic Stochastic PDEs}},
  {Preprint: arXiv:1806.06959 [math.ST]}, 2018.

\bibitem{Chong19}
\bysame, \emph{{High-Frequency Analysis of Parabolic Stochastic PDEs with
  Multiplicative Noise: Part I}}, {Preprint: arXiv:1908.04145 [math.PR]}, 2019.

\bibitem{Cialenco18}
I.~Cialenco, \emph{Statistical inference for {SPDE}s: an overview}, Stat.
  Inference Stoch. Process. \textbf{21} (2018), no.~2, 309--329.

\bibitem{CialencoDelgadoVencesKim19}
I.~Cialenco, F.~Delgado-Vences, and H.~Kim, \emph{{Drift Estimation for
  Discretely Sampled SPDEs}}, {Preprint: arXiv:1904.10884 [math.PR]}, 2019.

\bibitem{CialencoGlattHoltz11}
I.~Cialenco and N.~Glatt-Holtz, \emph{Parameter estimation for the
  stochastically perturbed {N}avier-{S}tokes equations}, Stochastic Process.
  Appl. \textbf{121} (2011), no.~4, 701--724.

\bibitem{CialencoHuang19}
I.~Cialenco and Y.~Huang, \emph{{A Note on Parameter Estimation for Discretely
  Sampled SPDEs}}, Stochastics and Dynamics (2019), forthcoming,
  doi:10.1142/S0219493720500161.

\bibitem{ConstantinFoias88}
P.~Constantin and C.~Foias, \emph{Navier-{S}tokes {E}quations}, Chicago
  Lectures in Mathematics, University of Chicago Press, 1988.

\bibitem{DaPratoZabczyk14}
G.~Da~Prato and J.~Zabczyk, \emph{{Stochastic Equations in Infinite
  Dimensions}}, second ed., Encyclopedia of Mathematics and its Applications,
  vol. 152, Cambridge University Press, 2014.

\bibitem{Fitzhugh61}
R.~Fitzhugh, \emph{{Impulses and Physiological States in Theoretical Models of
  Nerve Membrane}}, Biophys. J. \textbf{1} (1961), 445--466.

\bibitem{HubnerKhasminskiiRozovskii93}
M.~H\"{u}bner, R.~Khasminskii, and B.~L. Rozovskii, \emph{Two {E}xamples of
  {P}arameter {E}stimation for {S}tochastic {P}artial {D}ifferential
  {E}quations}, Stochastic processes (S.~Cambanis, J.~K. Ghosh, R.~L.
  Karandikar, and P.~K. Sen, eds.), Springer, New York, 1993, pp.~149--160.

\bibitem{Huebner93}
M.~Huebner, \emph{{Parameter Estimation for Stochastic Differential
  Equations}}, ProQuest LLC, Ann Arbor, 1993, Thesis (Ph.D.)--University of
  Southern California.

\bibitem{HuebnerLototskyRozovskii97}
M.~Huebner, S.~Lototsky, and B.~L. Rozovskii, \emph{Asymptotic {P}roperties of
  an {A}pproximate {M}aximum {L}ikelihood {E}stimator for {S}tochastic {PDE}s},
  Statistics and {C}ontrol of {S}tochastic {P}rocesses (Y.~M. Kabanov, B.~L.
  Rozovskii, and A.~N. Shiryaev, eds.), World Sci. Publ., 1997, pp.~139--155.

\bibitem{HuebnerRozovskii95}
M.~Huebner and B.~L. Rozovskii, \emph{On asymptotic properties of maximum
  likelihood estimators for parabolic stochastic {PDE}'s}, Probab. Theory
  Related Fields \textbf{103} (1995), no.~2, 143--163.

\bibitem{IbragimovHasminskii81}
I.~A. Ibragimov and R.~Z. Has'minskii, \emph{{Statistical Estimation
  (Asymptotic Theory)}}, Applications of Mathematics, vol.~16, Springer-Verlag,
  New York, 1981.

\bibitem{JacodShiryayev03}
J.~Jacod and A.~N. Shiryaev, \emph{{Limit Theorems for Stochastic Processes}},
  second ed., Grundlehren der Mathematischen Wissenschaften, vol. 288,
  Springer-Verlag, Berlin, Heidelberg, 2003.

\bibitem{Kutoyants04}
Y.~A. Kutoyants, \emph{{Statistical Inference for Ergodic Diffusion
  Processes}}, Springer Series in Statistics, Springer-Verlag London Ltd.,
  2004.

\bibitem{LionsMagenes72}
J.-L. Lions and E.~Magenes, \emph{{Non-Homogeneous Boundary Value Problems and
  Applications. Vol. I}}, {Die Grundlehren der Mathematischen Wissenschaften},
  vol. 181, Springer-Verlag, Berlin, Heidelberg, 1972.

\bibitem{LiptserShiryayev01}
R.~S. Liptser and A.~N. Shiryaev, \emph{Statistics of {R}andom {P}rocesses {II}
  ({A}pplications)}, second ed., Applications of Mathematics (Stochastic
  Modelling and Applied Probability), vol.~6, Springer-Verlag, Berlin,
  Heidelberg, 2001.

\bibitem{LiptserShiryayev77}
R.~S. Liptser and A.~N. Shiryayev, \emph{Statistics of {R}andom {P}rocesses {I}
  ({G}eneral {T}heory)}, {Applications of Mathematics}, vol.~5,
  Springer-Verlag, New York, 1977.

\bibitem{LiptserShiryayev89}
R.~Sh. Liptser and A.~N. Shiryayev, \emph{Theory of {M}artingales}, Mathematics
  and its Applications (Soviet Series), vol.~49, Kluwer Academic Publishers
  Group, Dordrecht, 1989.

\bibitem{LiuRockner15}
W.~Liu and M.~R\"{o}ckner, \emph{{Stochastic Partial Differential Equations: An
  Introduction}}, Universitext, Springer, 2015.

\bibitem{Lototsky03}
S.~Lototsky, \emph{{Parameter Estimation for Stochastic Parabolic Equations:
  Asymptotic Properties of a Two-Dimensional Projection-Based Estimator}},
  Stat. Inference Stoch. Process. \textbf{6} (2003), no.~1, 65--87.

\bibitem{Lototsky09}
S.~V. Lototsky, \emph{{Statistical Inference for Stochastic Parabolic
  Equations: A Spectral Approach}}, Publ. Mat. \textbf{53} (2009), no.~1,
  3--45.

\bibitem{LototskyRozovskii99}
S.~V. Lototsky and B.~L. Rosovskii, \emph{Spectral asymptotics of some
  functionals arising in statistical inference for {SPDE}s}, Stochastic
  Process. Appl. \textbf{79} (1999), no.~1, 69--94.

\bibitem{Nagumo62}
A.~Nagumo, S.~Arimoto, and S.~Yoshizawa, \emph{{An Active Pulse Transmission
  Line Simulating Nerve Axon}}, Proc. IRE \textbf{50} (1962), no.~10,
  2061--2070.

\bibitem{Peszat95}
S.~Peszat, \emph{{Existence and Uniqueness of the Solution for Stochastic
  Equations on Banach Spaces}}, Stochastics Stochastics Rep. \textbf{55}
  (1995), no.~3-4, 167--193.

\bibitem{PospisilTribe07}
Jan Posp\'{\i}\v{s}il and Roger Tribe, \emph{{Parameter Estimates and Exact
  Variations for Stochastic Heat Equations Driven by Space-Time White Noise}},
  Stoch. Anal. Appl. \textbf{25} (2007), no.~3, 593--611.

\bibitem{LototskyRozovskii00}
{S. Lototsky and B. L. Rozovskii}, \emph{{Parameter Estimation for Stochastic
  Evolution Equations with Non-Commuting Operators}}, {Skorokhod's ideas in
  probability theory} (V.~Korolyuk, N.~Portenko, and H.~Syta, eds.), {Institute
  of Mathematics of the National Academy of Science of Ukraine, Kiev}, 2000,
  pp.~271--280.

\bibitem{SauerStannat16}
M.~Sauer and W.~Stannat, \emph{Analysis and approximation of stochastic nerve
  axon equations}, Math. Comp. \textbf{85} (2016), no.~301, 2457--2481.

\bibitem{Shiryaev96}
A.~N. Shiryaev, \emph{Probability}, second ed., Graduate Texts in Mathematics,
  vol.~95, Springer-Verlag, New York, 1996.

\bibitem{Shubin01}
M.~A. Shubin, \emph{{Pseudodifferential Operators and Spectral Theory}}, second
  ed., Springer-Verlag, Berlin, Heidelberg, 2001.

\bibitem{VanDerVaart98}
A.~W. van~der Vaart, \emph{{Asymptotic Statistics}}, Cambridge Series in
  Statistical and Probabilistic Mathematics, vol.~3, Cambridge University
  Press, Cambridge, 1998.

\bibitem{Weyl11}
H.~Weyl, \emph{{{\"U}ber die asymptotische Verteilung der Eigenwerte}},
  {Nachrichten von der Gesellschaft der Wissenschaften zu G{\"o}ttingen,
  Mathematisch-Physikalische Klasse} (1911), {110--117}.

\end{thebibliography}

\end{document}